\documentclass[a4paper,11pt]{article}
\usepackage{stmaryrd}
\usepackage{graphicx}
\usepackage{amsfonts}
\usepackage{amsmath,amsthm}
\usepackage{amssymb}
\usepackage{mathrsfs}
\usepackage{indentfirst}
\usepackage{titlesec}
\usepackage{caption2}
\usepackage{color}
\usepackage[normalem]{ulem}
\usepackage{float}
\usepackage[english]{babel}
\usepackage{cite}
\usepackage{bbm}
\usepackage[colorlinks, linkcolor=blue]{hyperref}

\newtheorem{theorem}{Theorem}[section]
\newtheorem{definition}{Definition}[section]

\newtheorem{lemma}[theorem]{Lemma}
\newtheorem{corollary}[theorem]{Corollary}

\begin{document}

\title{On the long--time behavior of a perturbed conservative system with degeneracy.}

\author{Wenqing Hu
\thanks{Department of Mathematics and Statistics, Missouri University of Science and Technology
(formerly University of Missouri, Rolla). Email: \texttt{huwen@mst.edu}} \ .}

\date{}

\maketitle

\begin{abstract}
We consider in this work a model conservative system subject to dissipation and Gaussian--type stochastic perturbations.
The original conservative system possesses a continuous set of steady states, and is thus degenerate.
We characterize the long--time limit of our model system as the perturbation parameter tends to zero.
The degeneracy in our model system carries features found in some partial differential equations related,
for example, to turbulence problems.
\end{abstract}



\section{Introduction.}\label{Sec:Intro}

Many Hamiltonian systems that arise in mechanics, mechanical engineering, as well
as hydrodynamics are subject to group symmetry. As an example,
in the study of the motion of an ideal incompressible fluid, V.I.Arnold had proposed
(see \cite{[Arnold1966]}, \cite{[Arnold]}, \cite{[Arnold-Khesin]}, \cite{[TaoBlog]}) a beautiful picture that describes the
dynamics of ideal incompressible fluid as geodesic flows on the group of all diffeomorphisms of a certain domain
(see also the author's related work \cite{[Hu-Sverak]} in this direction).
The studies of random perturbations of Hamiltonian systems, or general dynamical systems with symmetry,
in particular the long--time dynamics and problems about invariant measures of these systems are of interest
(see also the author's related work \cite{[Hu-Metastability-Nearly-Elastic]}, \cite{[Freidlin-Hu-NearlyElastic]}, \cite{[Freidlin-HuLandau-Lifschitz]}).
Schematically, the general problem can be formulated as follows. We are given a dynamical
system
\begin{equation}\label{Intro:Eq:DynamicalSystem}
\dot{x}=b(x)
\end{equation}
in an ambient space $x\in M$
($M$ can be a Riemanian manifold). Usually we assume $b(x)$ preserves the energy. Then we assume
that for some group $G$ the
system \eqref{Intro:Eq:DynamicalSystem} has some symmetry with respect to
$G$. The last sentence about symmetry of the system
\eqref{Intro:Eq:DynamicalSystem} with respect to the group $G$ is a bit vague
and could be understood in many different ways. It can be understood
in a strict way so that the group can act on the space $M$ (in
particular, it is such case when $G=M$) and the dynamics of
\eqref{Intro:Eq:DynamicalSystem} is invariant with respect to $G$--action. It can
also be understood as a more ``rough" symmetry, in the sense for
example that the stable attractors of \eqref{Intro:Eq:DynamicalSystem} has
equivalent dynamical properties under $G$--action (in
\cite{[Freidlin2014JSPpaper]} such dynamical property is in the sense
of equivalence of logarithmic asymptotics of transition
probabilities when we add a small noise to \eqref{Intro:Eq:DynamicalSystem},
this is related to the notion of ``quasi--potential", see \cite{[FWbook]}, \cite{[FWbook2012]}).
Our goal is to describe
the effect of adding a small noise to \eqref{Intro:Eq:DynamicalSystem}. That
is, we study systems of type
\begin{equation}\label{Intro:Eq:DynamicalSystemNoise}
\dot{\mathcal{X}}^\varepsilon=b(\mathcal{X}^\varepsilon)+\xi^{\varepsilon}
\end{equation}
where $\xi^\varepsilon$ is a deterministic and/or stochastic
perturbation depending on
the small parameter(s) $\varepsilon=(\varepsilon_1,...,\varepsilon_k)$. Recent progresses in this direction have shown
that an effective description of the long--time behavior of
\eqref{Intro:Eq:DynamicalSystemNoise} is the motion on the cone of invariant
measures of the unperturbed system \eqref{Intro:Eq:DynamicalSystem}
(see \cite{[Freidlin2014JSPpaper]}). Several examples of such description are recently demonstrated in \cite{[Freidlin2014JSPpaper]},
\cite{[FKWtrap]}, \cite{[FKraretransition]}, \cite{[FKSlowlyChangingDynamics]}, \cite{[Freidlin-Koralov-WentzellPockets]}.

The above paradigm is only a general scheme. In this work we are interested in studying a model problem that falls
under the above general paradigm. Let us consider the following system
(see \cite{[BT]}, \cite[Section 4.4]{[BerglundKramers]}) corresponding to \eqref{Intro:Eq:DynamicalSystem}:

\begin{equation}\label{Intro:Eq:ABmodel}
\left\{\begin{array}{l} dx_t=-x_ty_tdt \ ,
\\
dy_t=x_t^2dt \ .
\end{array}\right.
\end{equation}

A phase picture of system (\ref{Intro:Eq:ABmodel}) can be seen in Figure \ref{Fig:ABModel}(a). We see
that the whole line $Oy_A$ contains stable equilibriums and the
whole line $Oy_B$ contains unstable equilibriums. This is different
from the cases considered in \cite{[FWbook]}, \cite{[FW1969]}. In
this case we can understand the symmetry of (\ref{Intro:Eq:ABmodel}) in a
more rough way: the stable and unstable equilibriums are symmetric
with respect to shifts in the directions of $Oy_A$ and $Oy_B$,
respectively. The unperturbed system (\ref{Intro:Eq:ABmodel}) preserves the
energy $E(x,y)=x^2+y^2$. The driving vector field $b(x,y)=(-xy,x^2)$ is degenerate on $x=0$.
Let us add a perturbation to system \eqref{Intro:Eq:ABmodel}
that consists of a deterministic friction and a random noise:

\begin{figure}
\centering
\includegraphics[height=7cm, width=11cm, bb=0 0 370 245]{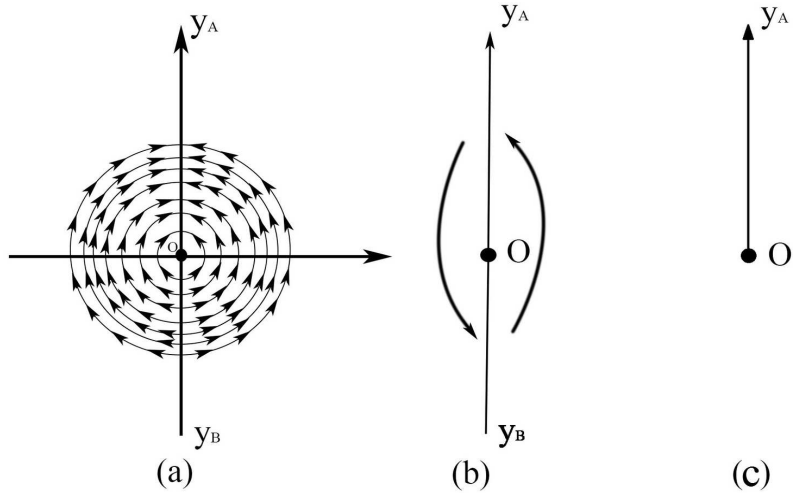}
\caption{The $AB$ model.}
\label{Fig:ABModel}
\end{figure}

\begin{equation}\label{Intro:Eq:ABmodelPerturbed}
\left\{\begin{array}{l} d\mathcal{X}^{\varepsilon}_t=-\mathcal{X}^{\varepsilon}_t \mathcal{Y}^{\varepsilon}_tdt-\varepsilon \mathcal{X}^{\varepsilon}_tdt+\sqrt{\varepsilon}dW^1_t \ , \ \mathcal{X}^\varepsilon_0=x_0 \ ,
\\
d\mathcal{Y}^{\varepsilon}_t=(\mathcal{X}^{\varepsilon}_t)^2dt-\varepsilon \mathcal{Y}^{\varepsilon}_tdt+\sqrt{\varepsilon}dW^2_t \ , \ \mathcal{Y}^\varepsilon_0=y_0 \ .
\end{array}\right.
\end{equation}

Here $W^1_t$ and $W^2_t$ are two independent standard $1$--dimensional Brownian motions; the small parameter $\varepsilon>0$
is the intensity of the friction, and the small parameter $\sqrt{\varepsilon}>0$ represents the intensity of the noise.
System \eqref{Intro:Eq:ABmodelPerturbed} is a two--dimensional nonlinear stochastic equation involving a
non--potential force. It is this non--potential force that has
the essential effect of creating a line of stable
fixed points (attracting line $Oy_A$) touching a line of unstable fixed points (repelling
line $Oy_B$). In the subsequent text we sometimes refer to this model as the $AB$--model.

Our goal in this paper is to study the long--time behavior of system \eqref{Intro:Eq:ABmodelPerturbed}
as $\varepsilon \downarrow 0$. By further developing results in \cite{[Freidlin2014JSPpaper]},
\cite{[FKWtrap]}, \cite{[FKraretransition]}, \cite{[FKSlowlyChangingDynamics]}, \cite{[Freidlin-Koralov-WentzellPockets]},
we will characterize the limiting process as a diffusion process $Y_t$ on the positive--$y$ semi--axis.
The limiting diffusion process $Y_t$ behaves as a $2$--dimensional radial Bessel process with linear damping,
and henceforce we call it a damped $2$--d radial Bessel process,
abbreviated as \textit{damped--BES(2)} (for Bessel process in arbitrary dimension
see \cite[Chapter XI, \S 1]{[Revuz-Yor]}).
The origin $O$ is an inaccessible point for damped--BES(2).
Diffusion processes on singular $1$--dimensional manifolds as the limit of averaging procedure
has been considered in \cite{[FW-DiffusionGraph]}, \cite{[FW-AMS]}, among many other literature.
The major contribution in our work is that we consider the manifold of unstable equilibria
touching the manifold of stable equilibria. This results in non--trivial analysis that leads to our limiting
process $Y_t$ as well as the inaccessibility of the origin $O$.
We will describe the limiting Markov diffusion process $Y_t$ via its infinitesimal generator,
and we show the weak convergence by making use of tightness and the classical martingale problem method.

In a certain sense, our model problem here differs from the set--up in the classical Freidlin--Wentzell theory
(see \cite{[FW1969]}) in that the point--like asymptotically stable attractor is replaced by a manifold.
We can view our limiting process $Y_t$ on $Oy_A$, the damped--BES(2) process,
as a ``process--level attractor" of our system. For $\varepsilon>0$, the dynamics of the system
as $\varepsilon\downarrow 0$ corresponds to the ``metastable" behavior (see \cite{[FreidlinSublimitingDistribution]}).
We will show that under this scenario the ``metastable" behavior of the system is characterized by
jumps between points on $Oy_A$ and $Oy_B$.

We are motivated by finite dimensional models for the inviscid stochastic 2--d Navier--Stokes equations written
in vorticity form
(see \cite{[Kuksin-Shirikyan2017Review]}, \cite[Lecture 39]{[SverakNotes2011-2012]})
\begin{equation}\label{Intro:Eq:InviscidStochasticNSE}
\dfrac{\partial \omega}{\partial t}+(u\cdot \nabla) \omega-\nu \Delta \omega=\sqrt{\nu}\eta(t,x) \ , \ u=\mathcal{K} \omega \ , \ \omega(0,x)=\omega_0(x) \ ,
\end{equation}
in which $\mathcal{K}=\nabla^\perp \Delta^{-1}$ is the Biot--Savart operator, $\eta(t,x)$ is a noise, and the viscosity parameter $\nu \rightarrow 0$.
An unsolved issue here targets at studying the vanishing noise
limit of stationary measures of the 2--d stochastic Navier--Stokes system
(see open problem 3 in the last section of the survey \cite{[Kuksin-Shirikyan2017Review]}).
The difficulty there is that one has to put a rather restrictive hypothesis, namely the unperturbed dynamics has to be
globally asymptotically stable. To remove this restriction, in the finite dimensional case this problem
is rather well--understood, and one can establish the so--called Freidlin--Wentzell asymptotics for stationary measures
(see Section 6.4 in \cite{[FWbook2012]}). As for stochastic PDEs, similar results can be proved, provided that
the global attractor for the unperturbed dynamics has a ``regular structure". The latter means that the attractor consists
 of finitely many steady--states and the heteroclinic orbits joining them. A result in this direction has been proved in
 \cite{[Martiosyan2017CPAM]} for the case of a damped nonlinear wave equation.
 However, the global attractor for the $2$--d Euler system does not have a regular structure, and in fact it
has continuous sets of steady states (see \cite[Lecture 68]{[SverakNotes2011-2012]}).
More generally, systems that arise in hydrodynamics, such as in the context of Euler's equation, typically possess
equilibrium points that belong to an infinite dimensional ``manifold" of other equilibria. These has been found in
experiments (see \cite{[2dTurbulence-Experiments-Sommeria]}, \cite{[2dTurbulence-Experiments-Tabling]}),
in numerical simulations (see \cite{[2dTurbulence-Numerical-Farge]}), explained
using arguments based on statistical mechanics
(see \cite{[2dTurbulence-StatMech-Venaille]}, \cite{[2dTurbulence-StatMech-Robert]}, \cite{[2dTurbulence-StatMech-Miller]},
\cite{[2dTurbulence-StatMech-Bouchet]}),
as well as explained theoretically (see \cite{[2dTurbulence-Theory-Morita]}, \cite{[2dTurbulence-Theory-MouhotVillani]}).
Our system \eqref{Intro:Eq:ABmodel} is a very simple finite--dimensional example of such type, in which the attractor
is a semi--line $Oy_A$. When we add a damping to \eqref{Intro:Eq:ABmodel}, we obtain for fixed $\varepsilon>0$
the model system \eqref{Intro:Eq:ABmodelPerturbed} without the stochastic noise, which admits only one single attractor $O$.
Of course, the situation will be much more complicated for the Euler and the Navier--Stokes equations.
For example, in low dimensions a good example is the famous Lorenz attractor (see \cite{[LorentzAttractor]}).
However, a surprising geometric connection is that our system \eqref{Intro:Eq:ABmodel}
can be viewed as the Euler--Arnold equation (see \cite{[TaoBlog]}, \cite[Appendix 2]{[Arnold]})
for the group of all affine transformations of a line $\ell$ (see
\cite{[Molchanov]} for more on this group), while the $2$--d Euler equation is the
Euler--Arnold equation for the group of all diffeomorphisms transforming the domain in which the fluid is moving
(see \cite{[Arnold1966]} and \cite[Appendix 2]{[Arnold]}). The formulation of our system
\eqref{Intro:Eq:ABmodel} as the Euler--Arnold equation will be discussed in Section \ref{Sec:Euler-Arnold}.

The paper is organized as follows. In Section \ref{Sec:Heuristic} we will explain the heuristics of the limiting mechanism.
In Section \ref{Sec:LimitingProcess} we demonstrate
the main convergence theorem as well as its proof. In Section \ref{Sec:AuxiliaryLemmas} we prove auxiliary lemmas that are needed in Section \ref{Sec:LimitingProcess}.
In Section \ref{Sec:Metastable} we describe the dynamics of our model system for small but nonzero $\varepsilon>0$.
In Section \ref{Sec:Euler-Arnold} we discuss the formulation of our system
\eqref{Intro:Eq:ABmodel} as the Euler--Arnold equation for the group of all affine transformations of a line.
Some remarks and generalizations are provided in Section \ref{Sec:Remark}.

\section{Heuristic description of the limiting mechanism.}\label{Sec:Heuristic}

To describe the limiting motion as $\varepsilon\downarrow 0$, we can first do a time rescaling $t\rightarrow \dfrac{t}{\varepsilon}$.
Let $(X^\varepsilon_t, Y^\varepsilon_t)=(\mathcal{X}^\varepsilon_{t/\varepsilon}, \mathcal{Y}^\varepsilon_{t/\varepsilon})$. Then we have

\begin{equation}\label{Eq:ABmodelPerturbedTimeRescaled}
\left\{\begin{array}{l}
dX^{\varepsilon}_t=-\dfrac{1}{\varepsilon}X^{\varepsilon}_t Y^{\varepsilon}_tdt- X^{\varepsilon}_tdt+dW^1_t \ , \ X^\varepsilon_0=x_0 \ ,
\\
dY^{\varepsilon}_t=\dfrac{1}{\varepsilon}(X^{\varepsilon}_t)^2dt- Y^{\varepsilon}_tdt+dW^2_t \ , \ Y^\varepsilon_0=y_0 \ .
\end{array}\right.
\end{equation}

In this way, we see the separation of a ``fast" motion which is governed by the non--potential force term, and a ``slow"
motion which is due to the random perturbation. Due to the effect of the fast motion, starting from anywhere
$(x_0, y_0)$ that is not lying on the semi--axis $Oy_B$,
the process $(X^\varepsilon_t, Y^\varepsilon_t)$ will come close to the attracting line $Oy_A$ in a relatively short time. Let $\pi$ denote
this hitting operator, so that we have the following definition.

\begin{definition}\label{Def:ProjectionOperatorPi}
We define a projection operator $\pi: \mathbb{R}^2\backslash Oy_A\rightarrow Oy_A$,
or equivalently $y^\pi(x_0,y_0): \mathbb{R}^2\backslash Oy_A\rightarrow \mathbb{R}_+$,
such that $\pi(x_0,y_0)=(0, y^\pi(x_0,y_0))$ as follows:
when $(x_0,y_0)\in \mathbb{R}^2\backslash (Oy_A\cup Oy_B)$, we set
$y^\pi(x_0,y_0)=\lim\limits_{t\rightarrow \infty}y(t)$ where $(x(t),y(t))$ is
the deterministic flow in \eqref{Intro:Eq:ABmodel}
 with initial condition $(x(0), y(0))=(x_0,y_0)$; when $(0,y_0)$ in $Oy_B$ (i.e. $y_0<0$),
we can then naturally extend the operator $\pi$ onto the $Oy_B$ axis, so that
$y^\pi(0,y_0)=y^\pi(|y_0|\sin\kappa, -|y_0|\cos\kappa)$
for some small $\kappa>0$; finally, we define $y^\pi(0,0)=0$.
\end{definition}

In the limit
as $\varepsilon \downarrow 0$, the process $(X^\varepsilon_t, Y^\varepsilon_t)$ is pushed by the flow onto $Oy_A$, and will be close to
$\pi(x_0, y_0)$ in short time.
There, the $Y$--component $Y^\varepsilon_t$ behaves
as a $2$--dimensional linearly damped radial Bessel process (damped--BES(2)) on $Oy_A$:

\begin{equation}\label{Eq:ABmodelYComponent}
dY_t=\left(\dfrac{1}{2Y_t}-Y_t\right)dt+dW^2_t \ , Y_0=y^\pi(x_0, y_0) \ .
\end{equation}

Indeed, when $Y_t$ is close to $O$, the large positive drift term $\dfrac{1}{2Y_t}$
comes from the limit of the positive drift $\dfrac{(X^\varepsilon_t)^2}{\varepsilon}$ in the $Y$--equation
of \eqref{Eq:ABmodelPerturbedTimeRescaled} as $\varepsilon\downarrow 0$ (which is illustrated as
Corollary \ref{Cor:HitClosenessToZeroSquareXOverEps:OutsideDelta}).
This makes the origin $O$ an inaccessible point for $Y_t$. However,
for small $\varepsilon>0$, the process $(X^\varepsilon_t, Y^\varepsilon_t)$ may still enter a thin strip around
the half--line $Oy_B$ through $O$. Due to the strong Markov property
of the process $(X^\varepsilon_t, Y^\varepsilon_t)$, once it enters the domain $\mathbb{R}^2\backslash Oy_B$,
it will move along the fast flow to hit somewhere on $Oy_A$. For any fixed $\varepsilon>0$, the probability
of hitting the level $Y=-a$ for some $a>0$ before moving along the fast flow and hit somewhere
on $Oy_A$ decays to $0$ as $\varepsilon\downarrow 0$. As the process $Y_t^\varepsilon$ is closer to the origin $O$,
the positive drift term $\dfrac{(X_t^\varepsilon)^2}{\varepsilon}$ pushes the process $Y_t^\varepsilon$
to bounce back to positive $y$--axis. Thus our limiting $Y$--process, the damped--BES(2),
only lives on the positive $Y$--axis (see Figure \ref{Fig:ABModel}(c)).

The above scenario can be roughly seen by considering the radial process $r_t^\varepsilon=\sqrt{(X_t^\varepsilon)^2+(Y_t^\varepsilon)^2}$.
In fact, by applying It\^{o}'s formula to \eqref{Eq:ABmodelPerturbedTimeRescaled} we see that

\begin{equation}\label{Eq:RadialProcessBessel}
\begin{array}{ll}
dr_t^\varepsilon & =\dfrac{X_t^\varepsilon}{r_t^\varepsilon}\left[\left(-\dfrac{1}{\varepsilon}X_t^\varepsilon Y_t^\varepsilon-X_t^\varepsilon\right)dt+dW_t^1\right]
\\
& \qquad \qquad +\dfrac{Y_t^\varepsilon}{r_t^\varepsilon}\left[\left(\dfrac{1}{\varepsilon}(X_t^\varepsilon)^2-Y_t^\varepsilon\right)dt+dW_t^2\right]
+\dfrac{1}{2}\dfrac{(Y_t^\varepsilon)^2}{(r_t^\varepsilon)^3}dt+\dfrac{1}{2}\dfrac{(X_t^\varepsilon)^2}{(r_t^\varepsilon)^3}dt
\\
& =\left(\dfrac{1}{2r_t^\varepsilon}-r_t^\varepsilon\right)dt+dW^r_t \ , \ r_0^\varepsilon=\sqrt{(X_0^\varepsilon)^2+(Y_0^\varepsilon)^2}
\end{array}
\end{equation}
where $W_t^r$ is a standard Brownian motion on $\mathbb{R}$.
When the process $X_t^\varepsilon$ is pushed by the flow to be close to the $Y$--axis, we have that
$Y_t^\varepsilon$ is close to $r_t^\varepsilon$, and thus \eqref{Eq:RadialProcessBessel} indicates
the limiting $Y$--dynamics \eqref{Eq:ABmodelYComponent}.

However, for fixed $\varepsilon>0$, at a subexponential time scale, excursions of the process $(X_t^\varepsilon, Y_t^\varepsilon)$
moving from $Oy_A$ towards a level set $y=-a$
will be observed. These excursions are directly crossing through a neighborhood of $O$. Due to the
repelling nature of $Oy_B$ and the random perturbation, the process will not strictly lie on $Oy_B$ but it
will move along the fast flow and come close to somewhere on $Oy_A$. This induces jumps from points in $Oy_B$ to
points in $Oy_A$ (see Figure \ref{Fig:ABModel}(b)). At even larger time scale, such as an exponentially long time scale,
large deviation effect makes
the process $(X_t^\varepsilon, Y_t^\varepsilon)$ move from the attracting line $Oy_A$ to the repelling line $Oy_B$. Such moves are not through
$O$ but are directed motions against the fast flow. Again the instability of $Oy_B$ and the random perturbation will make the process
quickly jump back to $Oy_A$. This induces back and forth jumps between points in $Oy_A$ and those in $Oy_B$ (see Figure \ref{Fig:ABModel}(b)).
As $\varepsilon$ becomes smaller, motions of the process $(X_t^\varepsilon, Y_t^\varepsilon)$ to $Oy_B$ and jumping back become more and more rare, and in the limit
no more such jumps appear, so that we come to the limiting process $Y_t$ which cannot penetrate through $O$.
Thus as $\varepsilon>0$ is close to $0$, the description of the ``metastable" behavior of system \eqref{Intro:Eq:ABmodelPerturbed}
involves both a diffusion part and a jump part.

Figure \ref{Fig:SamplePathABModel} shows sample pathes of the $X_t^\varepsilon$ and $Y_t^\varepsilon$ processes, as well as the limiting $Y$--process
(driven by the same Brownian motion as the driving Brownian motion for $Y_t^\varepsilon$)
starting from $(X,Y)=(0,2)$ in $15000$ steps for stepsize$=0.0001$, with all steps rescaled to $[0,1]$. In Figure \ref{Fig:SamplePathABModel}(a), (b),
the red curves are the sample pathes for $Y_t$, and the blue curves are for sample pathes
$Y_t^\varepsilon$ when $\varepsilon=0.1$ (Figure \ref{Fig:SamplePathABModel}(a)) and $\varepsilon=0.01$ (Figure \ref{Fig:SamplePathABModel}(b)). In Figure \ref{Fig:SamplePathABModel}(c), (d),
the black curves are the sample pathes for $X_t^\varepsilon$ when $\varepsilon=0.1$ (Figure \ref{Fig:SamplePathABModel}(c)) and $\varepsilon=0.01$ (Figure \ref{Fig:SamplePathABModel}(d)).
One can see that the process $X_t^\varepsilon$ is mainly localized near $0$, and the process $Y_t^\varepsilon$ behaves similarly as
the process $Y_t$, especially when the parameter $\varepsilon>0$ is small.

\begin{figure}
\centering
\includegraphics[height=10cm, width=12cm, bb=0 0 594 455]{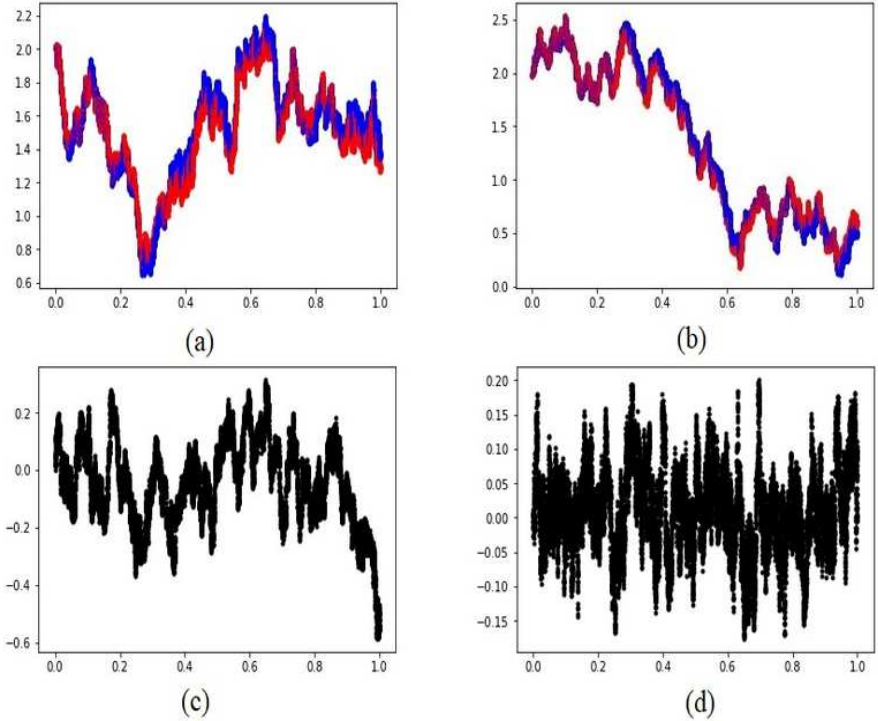}
\caption{Sample pathes of the $X_t^\varepsilon$ and $Y_t^\varepsilon$ processes, as well as the limiting $Y$--process (driven by $W_t^2$)
starting from $(X,Y)=(0,2)$ in $15000$ steps for stepsize$=0.0001$, that is rescaled to $[0,1]$. (a) $\varepsilon=0.1$; (b) $\varepsilon=0.01$;
the red curves are the sample pathes for $Y_t$, the blue curves are the sample pathes for
$Y_t^\varepsilon$. (c) $\varepsilon=0.1$; (d) $\varepsilon=0.01$;
the black curves are the sample pathes for $X_t^\varepsilon$.}
\label{Fig:SamplePathABModel}
\end{figure}

Let us also notice that, the cone formed by the set of extremal
invariant measures of the unperturbed system \eqref{Intro:Eq:ABmodel} consists of both the lines $Oy_A$ and $Oy_B$. And according to
\cite{[Freidlin2014JSPpaper]} the description of the limiting process shall be given by a Markov process on this cone.
Our result is in a sense a specific example of this general paradigm.
What we are demonstrating here is that the part $Oy_B$ of this cone is simply inaccessible, and the limiting process just lives on
$Oy_A$. This agrees with the heuristic that $Oy_A$ is the ``stable" half--line of equilibriums
and $Oy_B$ is the ``unstable" half--line of equilibriums.

\section{The limiting process and weak convergence theorem.}\label{Sec:LimitingProcess}

Let $Y_t$ be defined as the diffusion process on $\mathbb{R}$ with infinitesimal generator
given by the operator $A$ and domain of definition $D(A)$
(see \cite{[Dynkin1959]}).
For any continuous function $f: \mathbb{R} \rightarrow \mathbb{R}$ that is twice continuously
differentiable in $y\geq 0$ we have
\begin{equation}\label{Eq:LimitingGenerator:Bessel:OyA}
Af(y)=\dfrac{1}{2}\dfrac{d^2f}{d y^2}(y)+\left(\dfrac{1}{2y}-y\right)\dfrac{d f}{d y}(y) \ , \ \text{ for all } y>0 \ ,
\end{equation}
and
\begin{equation}\label{Eq:LimitingGenerator:Bessel:O}
Af(O)=\lim\limits_{y\rightarrow 0+}Af(y) \ .
\end{equation}
For $y<0$ we further define
\begin{equation}\label{Eq:LimitingGenerator:Bessel:OyB}
Af(y)=0 \text{ for all } y<0 \ .
\end{equation}

The domain of definition
of the operator $A$ is given by the set of continuous functions $f: \mathbb{R} \rightarrow \mathbb{R}$ such that $f(y)$
are twice continuously differentiable
in $y\geq 0$, with the limit of $\dfrac{d^+ f}{d y}(y)=\lim\limits_{z\rightarrow y, z>y}\dfrac{f(z)-f(y)}{z-y}$ exist
and is equal to zero as $y\rightarrow 0+$, i.e.
\begin{equation}\label{Eq:LimitingGenerator:Domain}
\lim\limits_{y\rightarrow 0+}\dfrac{d^+ f}{d y}(y)=0 \ .
\end{equation}

By \eqref{Eq:LimitingGenerator:Bessel:O} and \eqref{Eq:LimitingGenerator:Domain} we infer further that
\begin{equation}\label{Eq:LimitingGenerator:fPrimeOver-y}
\lim\limits_{y\rightarrow 0+}\dfrac{1}{y}\dfrac{d^+ f}{d y}(y)
\end{equation}
exists.

The existence of such a process $Y_t$ is guaranteed by the Hille--Yosida theorem (see \cite{[Feller1957]}, \cite{[Mandl]}).
The closure $\overline{A|_{D(A)}}$ of the operator $A$ in the space of continuous functions on $\mathbb{R}$
exists and it actually defines a Markov process on $\{y\geq 0\}$, which is a $2$--dimensional radial Bessel process
with linear damping on $\mathbb{R}_+$, that is
inaccessible to the origin $O$, and it contains isolated points on $\{y<0\}$.
Our main theorem can be stated as follows.

\begin{theorem}\label{Thm:MainTheorem}
Let $T>0$ and initial condition $(x_0, y_0)\in \mathbb{R}^2$. Then

(a) For any bounded continuous function $F:\mathbb{R}^2\rightarrow \mathbb{R}$ that is uniformly
Lipschitz continuous with a Lipschitz constant $\text{Lip}(F)<\infty$ we have
\begin{equation}\label{Thm:MainTheorem:Eq:X-WeakConvergenceTo0}
\lim\limits_{\varepsilon\downarrow 0}\mathbf{E} \left[F(X_T^\varepsilon, Y_T^\varepsilon)-F(0, Y_T^\varepsilon)\right]=0 \ .
\end{equation}

(b) The measures on $\mathbf{C}_{[0,T]}(\mathbb{R})$ induced by the
process $Y^\varepsilon_t$ converge weakly as $\varepsilon\downarrow 0$ to the measure induced by $Y_t$ with $Y_0=y^\pi(x_0, y_0)$.
\end{theorem}

\begin{proof}
Let $\delta=\delta(\varepsilon)=\varepsilon^{\alpha}>0$ with $\delta \rightarrow 0$ as $\varepsilon \downarrow 0$. We pick $\alpha=\dfrac{1}{10}$. Set $\sigma_0=0$ and
$$\tau_k=\inf\{t\geq \sigma_{k-1}, |Y_t^\varepsilon|=\delta\} \ , \ \sigma_k=\inf\{t\geq \tau_k, |Y_t^\varepsilon|=2\delta\} \ , \ k=1,2,... .$$

Our proof intuitively goes as follows:

Step 1. We show that if $Y_t^\varepsilon\geq \delta$, then as $\varepsilon\downarrow 0$ the process $X_t^\varepsilon$ is very close to the $Y$--axis.
This is proved in Lemma \ref{Lm:HitClosenessToZeroX:OutsideDelta}. We then show in Lemma \ref{Lm:HitClosenessToZeroSquareXOverEps:OutsideDelta}
and Corollary \ref{Cor:HitClosenessToZeroSquareXOverEps:OutsideDelta}
that as $X_t^\varepsilon$ is small, the quantity $\dfrac{(X_t^\varepsilon)^2}{\varepsilon}$ is close to $\dfrac{1}{2Y_t^\varepsilon}$.
In particular, this makes the process $Y_t^\varepsilon$ behaves close to a $2$--dimensional radial Bessel process
with linear damping when $Y_t^\varepsilon\geq \delta$.

Step 2. We show that during the time $\tau_k\leq t\leq \sigma_k$ we have $|X_t^\varepsilon|\leq 3\delta$ with high probability.
This is because whenever $|X_t^\varepsilon|\geq 2\delta$ the flow \eqref{Eq:ABmodelPerturbedTimeRescaled}
with small $\varepsilon>0$ will quickly bring the particle back to the region $Y\geq \delta$, and during this process
the $|X|$--value is less or equal than $3\delta$. This is done in Lemma
\ref{Lm:HitBackToZeroX:AbsXBiggerThanHalfKappa}.

Step 3. We show that $\mathbf{P}(Y^\varepsilon_{\sigma_k}=2\delta)\rightarrow 1$ as $\varepsilon\downarrow 0$ and therefore $\delta(\varepsilon)\rightarrow 0$.
This is because if $Y_t^\varepsilon\leq -1.5\delta$, then the flow of \eqref{Eq:ABmodelPerturbedTimeRescaled} with small $\varepsilon>0$
will quickly bring the particle back to $Y\geq \delta$, and during this process the $Y$--coordinate is $\geq -1.99\delta$ with probability $\rightarrow 1$
as $\varepsilon \downarrow 0$. This is done in Lemma \ref{Lm:HitBackToZeroX:YLessThanMinusDelta}.

Step 4. We then estimate $\mathbf{E}(\sigma_k-\tau_k)\lesssim \mathcal{O}(\delta^2)$ in Lemma \ref{Lm:ExitTimeEstimate:NearO:TauToSigma}. By making use of the fact that $|X_t^\varepsilon|$
will be close to $0$ for $\sigma_k\leq t\leq \tau_{k+1}$, we estimate $\mathbf{E}(\tau_{k+1}-\sigma_k)\gtrsim \mathcal{O}(\delta)\rightarrow 0$ as $\varepsilon\downarrow 0$
in Lemma \ref{Lm:ExitTimeEstimate:FarFromO:SigmaToTau}. The asymptotic lower bound for $\mathbf{E}(\tau_{k+1}-\sigma_k)$ provides us with an upper bound on the
number of up--crossings $N(\varepsilon)\lesssim \mathcal{O}(\delta^{-1})$ from $\delta$ to $2\delta$ before time $T$. This is done in Lemma \ref{Lm:UpperBoundNumberOfCrossingsBeforeT}.
Combining Lemmas \ref{Lm:UpperBoundNumberOfCrossingsBeforeT} and \ref{Lm:ExitTimeEstimate:NearO:TauToSigma} we obtain that
$N(\varepsilon)\cdot \mathbf{E}(\sigma_k-\tau_k) \rightarrow 0$ as $\varepsilon\downarrow 0$.

Steps 1 and 2 together help us to settle \eqref{Thm:MainTheorem:Eq:X-WeakConvergenceTo0} so part (a) of this Theorem.
To prove part (b) of this Theorem, we shall make use of a modification of Lemma 3.1 in \cite[Chapter 8]{[FWbook]}.
This has been used in the works \cite{[FKWtrap]}, \cite{[Freidlin-Koralov-WentzellPockets]}, \cite{[FW-DiffusionGraph]},
\cite{[DK-AveragingErgodic-AOP]}, \cite{[DK-AveragingErgodic-JAMS]}, \cite{[Freidlin-Hu-Wentzell-SPA]}, \cite{[Freidlin-Hu-CPDE]}.
First, in Lemma \ref{Lm:TightnessYProcess} we show that the family of processes $Y^\varepsilon_t$ is tight in $\mathbf{C}_{[0,T]}(\mathbb{R})$.
Secondly, we show that for every continuous function $f: \mathbb{R}_+\rightarrow \mathbb{R}$ such that $f\in D(A)$ and every $T>0$, having bounded derivatives
up to the third order, uniformly in the initial condition $(x_0,y_0)\in \mathbb{R}^2$ we have
\begin{equation}\label{Thm:MainTheorem:Eq:ConvergenceMartingaleProblem}
\mathbf{E}_{(x_0,y_0)}\left[f(Y_T^\varepsilon)-f(y^\pi(X_0^\varepsilon, Y_0^\varepsilon))-\int_0^T Af(Y_t^\varepsilon)dt\right]\rightarrow 0
\end{equation}
as $\varepsilon\downarrow 0$. The desired convergence in (b) then follows from \eqref{Thm:MainTheorem:Eq:ConvergenceMartingaleProblem}
by the argument using martingale problem formulation of Markov processes
(see \cite[Chapter 4]{[Ethier-Kurtz]}). We are left with proving \eqref{Thm:MainTheorem:Eq:ConvergenceMartingaleProblem}. To this end, we decompose
$$\begin{array}{ll}
& \displaystyle{\mathbf{E}\left[f(Y_T^\varepsilon)-f(y^\pi(X_0^\varepsilon, Y_0^\varepsilon))-\int_0^T Af(Y_t^\varepsilon)dt\right]}
\\
=& \displaystyle{\sum\limits_{k=1}^N \mathbf{E}\left[f(Y_{\tau_k}^\varepsilon)-f(Y_{\sigma_{k-1}}^\varepsilon)-\int_{\sigma_{k-1}}^{\tau_k} Af(Y_t^\varepsilon)dt\right]
+\sum\limits_{k=1}^N \mathbf{E}\left[f(Y_{\sigma_k}^\varepsilon)-f(Y_{\tau_k}^\varepsilon)-\int_{\tau_k}^{\sigma_k} Af(Y_t^\varepsilon)dt\right]}
\\
& \displaystyle{\qquad \qquad - \mathbf{E}\left[f(Y_{\sigma_N}^\varepsilon)-f(Y_T^\varepsilon)-\int_T^{\sigma_N} Af(Y_t^\varepsilon)dt\right]}
\\
= & (I)+(I\!I)-(I\!I\!I) \ .
\end{array}$$

Let us first estimate $(I)$. In fact, we can estimate, by Lemma \ref{Lm:HitClosenessToZeroSquareXOverEps:OutsideDelta}, that
$$\left|\mathbf{E}\left[f(Y_{\tau_k}^\varepsilon)-f(Y_{\sigma_{k-1}}^\varepsilon)-\int_{\sigma_{k-1}}^{\tau_k} Af(Y_t^\varepsilon)dt\right]\right|\leq
 CT\varepsilon^{1-4\alpha} \ .$$

This helps us to conclude, by further making use of Lemma \ref{Lm:UpperBoundNumberOfCrossingsBeforeT} that

$$\left|\sum\limits_{k=1}^N \mathbf{E}\left[f(Y_{\tau_k}^\varepsilon)-f(Y_{\sigma_{k-1}}^\varepsilon)-\int_{\sigma_{k-1}}^{\tau_k} Af(Y_t^\varepsilon)dt\right]\right|\leq
CT^2\varepsilon^{1-5\alpha} \rightarrow 0$$
as $\varepsilon\downarrow 0$, for $0<\alpha<\dfrac{1}{5}$ say $\alpha=\dfrac{1}{10}$.

To estimate $(I\!I)$, we notice that $Y_{\sigma_k}^\varepsilon=2\delta$ and $Y_{\tau_k}^\varepsilon=\delta$. Thus by using the fact that $f'(0)=0$ we obtain

$$f(Y_{\sigma_k}^\varepsilon)-f(0)\approx 4f''(0)\delta^2+\mathcal{O}(\delta^3) \ , \ f(Y_{\tau_k}^\varepsilon)-f(0)\approx f''(0)\delta^2+\mathcal{O}(\delta^3) \ ,$$
so that
$$\left|\mathbf{E}\left[f(Y_{\sigma_k}^\varepsilon)-f(Y_{\tau_k}^\varepsilon)-\int_{\tau_k}^{\sigma_k} Af(Y_t^\varepsilon)dt\right]\right|
\leq C_1|f''(0)|\delta^2+C_2\mathbf{E}(\sigma_k-\tau_k) \ .$$
This combined with the fact that $N(\varepsilon)\cdot \mathbf{E}(\sigma_k-\tau_k)\rightarrow 0$ and $N\cdot \delta^2\rightarrow 0$ as $\varepsilon\downarrow 0$ from Lemmas \ref{Lm:UpperBoundNumberOfCrossingsBeforeT} and \ref{Lm:ExitTimeEstimate:NearO:TauToSigma}, help us to conclude that
$|(I\!I)|\rightarrow 0$ as $\varepsilon\downarrow 0$.

Finally it is easy to see that $|(I\!I\!I)|\rightarrow 0$ as $\varepsilon \downarrow 0$. Thus \eqref{Thm:MainTheorem:Eq:ConvergenceMartingaleProblem} is proved.
\end{proof}

\section{Proof of auxiliary lemmas.}\label{Sec:AuxiliaryLemmas}

Recall that by \eqref{Eq:ABmodelPerturbedTimeRescaled}, we have
$$\begin{array}{lll}
dX_t^\varepsilon=\left(-\dfrac{1}{\varepsilon}X_t^\varepsilon Y_t^\varepsilon-X_t^\varepsilon\right)dt+dW_t^1 & , & X_0^\varepsilon=x_0 \ ,
\\
dY_t^\varepsilon=\left(\dfrac{1}{\varepsilon}(X_t^\varepsilon)^2-Y_t^\varepsilon\right)dt+dW_t^2 & , & Y_0^\varepsilon=y_0 \ .
\end{array}$$

\begin{lemma}\label{Lm:HitClosenessToZeroX:OutsideDelta}
For any $\delta=\delta(\varepsilon)$ such that $\delta=\varepsilon^{\alpha} \rightarrow 0$ as $\varepsilon\downarrow 0$ for some $0<\alpha<1$,
there exist some $t_0=t_0(\varepsilon)$ which can be picked as $t_0(\varepsilon)=\varepsilon^{(1-\alpha)/2}$, such that
as $t\geq t_0(\varepsilon)$ and while $Y_s^\varepsilon\geq \delta$ for $0\leq s \leq t$, we have
\begin{equation}\label{Lm:HitClosenessToZeroX:OutsideDelta:Eq:EstimateOfSquareX}
\mathbf{E} (X_t^\varepsilon)^2\leq C\varepsilon^{1-\alpha}
\end{equation}
for some $C>0$.
\end{lemma}

\begin{proof}
Let $Y_s^\varepsilon\geq \delta$ for $0\leq s\leq t$. Let $s\in [0,t]$ and we
consider applying It\^{o}'s formula to $(X_s^\varepsilon)^2$. In this way, we obtain from
\eqref{Eq:ABmodelPerturbedTimeRescaled} that

\begin{equation}\label{Lm:HitClosenessToZeroX:OutsideDelta:Eq:ItoXSquare}
\begin{array}{ll}
d(X_s^\varepsilon)^2 & =2X_s^\varepsilon dX_s^\varepsilon+(dX_s^\varepsilon)^2
\\
& =2\left(-\dfrac{Y_s^\varepsilon}{\varepsilon}-1\right)(X_s^\varepsilon)^2ds+2X_s^\varepsilon dW_s^1+ds \ .
\end{array}
\end{equation}

Therefore taking expectation in \eqref{Lm:HitClosenessToZeroX:OutsideDelta:Eq:ItoXSquare} we obtain

\begin{equation}\label{Lm:HitClosenessToZeroX:OutsideDelta:Eq:ItoXSquareExpectation}
d\mathbf{E}(X_s^\varepsilon)^2=2\mathbf{E}\left(-\dfrac{Y_s^\varepsilon}{\varepsilon}-1\right)(X_s^\varepsilon)^2ds+ds \ .
\end{equation}

As we have $Y_s^\varepsilon\geq \delta$ and $(X_s^\varepsilon)^2\geq 0$, we can estimate

$$\left(-\dfrac{Y_s^\varepsilon}{\varepsilon}-1\right)(X_s^\varepsilon)^2\leq \left(-\dfrac{\delta}{\varepsilon}-1\right)(X_s^\varepsilon)^2 \ ,$$
so that \eqref{Lm:HitClosenessToZeroX:OutsideDelta:Eq:ItoXSquareExpectation} becomes

$$d\mathbf{E}(X_s^\varepsilon)^2\leq 2\left(-\dfrac{\delta}{\varepsilon}-1\right)\mathbf{E}(X_s^\varepsilon)^2ds+ds \ .$$

Thus
$$\begin{array}{ll}
d\left[e^{2(\frac{\delta}{\varepsilon}+1)s}\mathbf{E}(X_s^\varepsilon)^2\right] & \leq e^{2(\frac{\delta}{\varepsilon}+1)s}\left(2\left(-\dfrac{\delta}{\varepsilon}-1\right)\mathbf{E}(X_s^\varepsilon)^2ds+ds+
2\left(\dfrac{\delta}{\varepsilon}+1\right)\mathbf{E}(X_s^\varepsilon)^2ds\right)
\\
& = e^{2(\frac{\delta}{\varepsilon}+1)s}ds \ .
\end{array}$$

Integrating the above differential inequality in the argument $s$ from $0$ to $t$ we see that we have
$$e^{2(\frac{\delta}{\varepsilon}+1)t}\mathbf{E}(X_t^\varepsilon)^2-\mathbf{E}(X_0^\varepsilon)^2\leq \dfrac{1}{2(\frac{\delta}{\varepsilon}+1)}\left(e^{2(\frac{\delta}{\varepsilon}+1)t}-1\right) \ ,$$
i.e.
$$\mathbf{E} (X_t^\varepsilon)^2\leq e^{-2(\frac{\delta}{\varepsilon}+1)t}\mathbf{E}(X_0^\varepsilon)^2+\dfrac{1}{2(\frac{\delta}{\varepsilon}+1)}(1-e^{-2(\frac{\delta}{\varepsilon}+1)t}) \ .$$
So finally we obtain the estimate

\begin{equation}\label{Lm:HitClosenessToZeroX:OutsideDelta:Eq:XSquareExpectationBound}
\mathbf{E} (X_t^\varepsilon)^2\leq e^{-2(\frac{\delta}{\varepsilon}+1)t}\mathbf{E}(X_0^\varepsilon)^2+\dfrac{1}{2(\frac{\delta}{\varepsilon}+1)} \ .
\end{equation}

As we have $\delta=\varepsilon^{\alpha}$, the above estimate \eqref{Lm:HitClosenessToZeroX:OutsideDelta:Eq:XSquareExpectationBound}
implies that we have

$$\mathbf{E} (X_t^\varepsilon)^2\leq e^{-2(\varepsilon^{-(1-\alpha)}+1)t}\mathbf{E}(X_0^\varepsilon)^2+\dfrac{1}{2}\varepsilon^{1-\alpha} \ .$$
From here we infer that as $t\geq t_0(\varepsilon)$ and $\varepsilon>0$ sufficiently small we have
$$\mathbf{E} (X_t^\varepsilon)^2\leq C\varepsilon^{1-\alpha}$$
for some $C>0$. In particular, we can pick $t_0(\varepsilon)=\varepsilon^{(1-\alpha)/2}$.
\end{proof}

The above estimate \eqref{Lm:HitClosenessToZeroX:OutsideDelta:Eq:EstimateOfSquareX} cannot provide a precise estimate
for $\dfrac{(X_t^\varepsilon)^2}{\varepsilon}$, which enters as the first term in the right--hand side of the equation for $Y_t^\varepsilon$.
In fact, this estimate can be obtained by first noticing the following Lemma.

\begin{lemma}\label{Lm:HitClosenessToZeroSquareXOverEps:OutsideDelta}
There exist some constant $C>0$ so that for small $\varepsilon>0$ and any function $f\in D(A)$
with bounded derivatives up to third order, uniformly in $k=1,2,...,N$ we have
\begin{equation}\label{Lm:HitClosenessToZeroSquareXOverEps:OutsideDelta:Eq:MartingaleErrorY-r}
\left|\mathbf{E}\left[f(Y_{\tau_k}^\varepsilon)-f(Y_{\sigma_{k-1}}^\varepsilon)-\displaystyle{\int_{\sigma_{k-1}}^{\tau_k}Af(Y_t^\varepsilon)dt}\right]\right|\leq CT\varepsilon^{1-4\alpha} \ .
\end{equation}
Here the constant $C>0$ depends on the bounds for the derivatives of $f$.
\end{lemma}

\begin{proof}
Let us assume that there exist uniform constants $M_1>0$, $M_2>0$ and $M_3>0$ such that
$|f'(y)|\leq M_1$, $|f''(y)|\leq M_2$ and $|f'''(y)|\leq M_3$. In fact, as
$Y_{\tau_k}^\varepsilon=\delta$ and $Y_{\sigma_k}^\varepsilon=2\delta$ for $k=1,2,...,N$, we have

\begin{equation}\label{Lm:HitClosenessToZeroSquareXOverEps:OutsideDelta:Eq:MartingaleErrorY-r-1}
\begin{array}{ll}
& \left|\mathbf{E}\left[f(Y_{\tau_k}^\varepsilon)-f(Y_{\sigma_{k-1}}^\varepsilon)-\displaystyle{\int_{\sigma_{k-1}}^{\tau_k}Af(Y_t^\varepsilon)dt}\right]
-\mathbf{E}\left[f(r_{\tau_k}^\varepsilon)-f(r_{\sigma_{k-1}}^\varepsilon)-\displaystyle{\int_{\sigma_{k-1}}^{\tau_k}Af(r_t^\varepsilon)dt}\right]\right|
\\
\leq &
M_1 \left[\mathbf{E}|Y_{\tau_k}^\varepsilon-r_{\tau_k}^\varepsilon|+\mathbf{E}|Y_{\sigma_{k-1}}^\varepsilon-r_{\sigma_{k-1}}^\varepsilon|\right]
\\
& \ + \mathbf{E}\left|\displaystyle{\int_{\sigma_{k-1}}^{\tau_k}}
\left[\left(\dfrac{1}{2Y_t^\varepsilon}-Y_t^\varepsilon\right)f'(Y_t^\varepsilon)+\dfrac{1}{2}f''(Y_t^\varepsilon)\right]dt
- \displaystyle{\int_{\sigma_{k-1}}^{\tau_k}}
\left[\left(\dfrac{1}{2r_t^\varepsilon}-r_t^\varepsilon\right)f'(r_t^\varepsilon)+\dfrac{1}{2}f''(r_t^\varepsilon)\right]dt\right|
\\
= & M_1 \left[\mathbf{E}|Y_{\tau_k}^\varepsilon-r_{\tau_k}^\varepsilon|+\mathbf{E}|Y_{\sigma_{k-1}}^\varepsilon-r_{\sigma_{k-1}}^\varepsilon|\right]
\\
& \ + \mathbf{E}\left|\displaystyle{\int_{\sigma_{k-1}}^{\tau_k}}
\left[\left(\dfrac{1}{2Y_t^\varepsilon}-Y_t^\varepsilon\right)f'(Y_t^\varepsilon)+\dfrac{1}{2}f''(Y_t^\varepsilon)\right]dt
- \displaystyle{\int_{\sigma_{k-1}}^{\tau_k}}
\left[\left(\dfrac{1}{2r_t^\varepsilon}-r_t^\varepsilon\right)f'(r_t^\varepsilon)+\dfrac{1}{2}f''(r_t^\varepsilon)\right]dt\right|
\\
= & M_1 \left[\mathbf{E}|Y_{\tau_k}^\varepsilon-r_{\tau_k}^\varepsilon|+\mathbf{E}|Y_{\sigma_{k-1}}^\varepsilon-r_{\sigma_{k-1}}^\varepsilon|\right]
\\
& \ + \mathbf{E}\left|\displaystyle{\int_{\sigma_{k-1}}^{\tau_k}}
\left[\left(\dfrac{1}{2Y_t^\varepsilon}-Y_t^\varepsilon\right)\left(f'(Y_t^\varepsilon)-f'(r_t^\varepsilon)\right)
+\left(\dfrac{1}{2Y_t^\varepsilon}-Y_t^\varepsilon-\dfrac{1}{2r_t^\varepsilon}+r_t^\varepsilon\right)f'(r_t^\varepsilon)\right.\right.
\\
& \left.\left. \qquad \qquad \qquad +\dfrac{1}{2}\left(f''(Y_t^\varepsilon)-f''(r_t^\varepsilon)\right)\right]dt\right|
\\
\leq & C M_1 \left[\mathbf{E}|Y_{\tau_k}^\varepsilon-r_{\tau_k}^\varepsilon|+\mathbf{E}|Y_{\sigma_{k-1}}^\varepsilon-r_{\sigma_{k-1}}^\varepsilon|\right]
+M_2\left(\dfrac{1}{\delta}+\delta\right)\mathbf{E}\displaystyle{\int_{\sigma_{k-1}}^{\tau_k}}|Y_t^\varepsilon-r_t^\varepsilon|dt
\\
& \qquad \qquad \qquad +M_1\left(1+\dfrac{1}{\delta^2}\right)\mathbf{E}\displaystyle{\int_{\sigma_{k-1}}^{\tau_k}}|Y_t^\varepsilon-r_t^\varepsilon|dt
+M_3\mathbf{E}\displaystyle{\int_{\sigma_{k-1}}^{\tau_k}}|Y_t^\varepsilon-r_t^\varepsilon|dt \ .
\end{array}
\end{equation}

As we have $Y_t^\varepsilon\geq \delta$ and $r_t^\varepsilon=\sqrt{(X_t^\varepsilon)^2+(Y_t^\varepsilon)^2}\geq Y_t^\varepsilon\geq \delta$
for $\sigma_{k-1}\leq t\leq \tau_k$, we infer that
\begin{equation}\label{Lm:HitClosenessToZeroSquareXOverEps:OutsideDelta:BoundY-rInX}
|Y_{\tau_k}^\varepsilon-r_{\tau_k}^\varepsilon|\leq \dfrac{|(Y_t^\varepsilon)^2-(r_t^\varepsilon)^2|}{|Y_t^\varepsilon+r_t^\varepsilon|}\leq \dfrac{1}{2\delta}(X_t^\varepsilon)^2 \ .
\end{equation}

From \eqref{Lm:HitClosenessToZeroSquareXOverEps:OutsideDelta:Eq:MartingaleErrorY-r-1} and
\eqref{Lm:HitClosenessToZeroSquareXOverEps:OutsideDelta:BoundY-rInX}, taking into account
Lemma \ref{Lm:HitClosenessToZeroX:OutsideDelta}, we know that, as $\varepsilon>0$ is small, for some constant $M>0$
we have

\begin{equation}\label{Lm:HitClosenessToZeroSquareXOverEps:OutsideDelta:Eq:MartingaleErrorY-r-2}
\begin{array}{ll}
& \left|\mathbf{E}\left[f(Y_{\tau_k}^\varepsilon)-f(Y_{\sigma_{k-1}}^\varepsilon)-\displaystyle{\int_{\sigma_{k-1}}^{\tau_k}Af(Y_t^\varepsilon)dt}\right]
-\mathbf{E}\left[f(r_{\tau_k}^\varepsilon)-f(r_{\sigma_{k-1}}^\varepsilon)-\displaystyle{\int_{\sigma_{k-1}}^{\tau_k}Af(r_t^\varepsilon)dt}\right]\right|
\\
\leq & C M_1 \left[\dfrac{1}{2\delta}\mathbf{E}(X_{\tau_k}^\varepsilon)^2+\dfrac{1}{2\delta}\mathbf{E}(X_{\sigma_{k-1}}^\varepsilon)^2\right]
\\
& \qquad \qquad \qquad +\left[M_2\left(\dfrac{1}{\delta}+\delta\right)+
M_1\left(1+\dfrac{1}{\delta^2}\right)\dfrac{1}{2\delta}+M_3\dfrac{1}{2\delta}
\right]\cdot \mathbf{E}\displaystyle{\int_{\sigma_{k-1}}^{\tau_k}}(X_t^\varepsilon)^2dt
\\
\leq & \dfrac{M}{\delta} \left[\mathbf{E}(X_{\tau_k}^\varepsilon)^2+\mathbf{E}(X_{\sigma_{k-1}}^\varepsilon)^2\right]
+\dfrac{M}{\delta^3} \cdot \displaystyle{\int_0^T \mathbf{E}(X_t^\varepsilon)^2dt}
\\
\leq & C[\varepsilon^{1-2\alpha}+T\varepsilon^{1-4\alpha}]\leq CT \varepsilon^{1-4\alpha}  \ .
\end{array}
\end{equation}

As we have, by martingale formulation of Markov processes, that
$$\mathbf{E}\left[f(r_{\tau_k}^\varepsilon)-f(r_{\sigma_{k-1}}^\varepsilon)-\displaystyle{\int_{\sigma_{k-1}}^{\tau_k}Af(r_t^\varepsilon)dt}\right]=0 \ ,$$
we see that the claim \eqref{Lm:HitClosenessToZeroSquareXOverEps:OutsideDelta:Eq:MartingaleErrorY-r} follows from
\eqref{Lm:HitClosenessToZeroSquareXOverEps:OutsideDelta:Eq:MartingaleErrorY-r-2}.
\end{proof}

\begin{corollary}\label{Cor:HitClosenessToZeroSquareXOverEps:OutsideDelta}
For any $1\leq k\leq N$ and any $\sigma_{k-1}\leq t_1\leq t_2\leq \tau_k$, we have
\begin{equation}\label{Cor:HitClosenessToZeroSquareXOverEps:OutsideDelta:WeakEstimateSquareXOverEps}
\left|\mathbf{E}\int_{t_1}^{t_2} \left(\dfrac{1}{\varepsilon}(X_t^\varepsilon)^2-\dfrac{1}{2Y_t^\varepsilon}\right)dt\right|
\leq C[\varepsilon^{1-4\alpha}+\varepsilon^{1-6\alpha}(t_2-t_1)] \ .
\end{equation}
\end{corollary}

\begin{proof}
Let us consider a function $f\in D(A)$ having bounded derivatives up to the third order.
 We can apply It\^{o}'s formula to the $Y$--dynamics of
\eqref{Eq:ABmodelPerturbedTimeRescaled} and we obtain, for any $\sigma_{k-1}\leq t_1\leq t_2\leq \tau_k$, that
$$\begin{array}{ll}
&f(Y_{t_2}^\varepsilon)-f(Y_{t_1}^\varepsilon)
\\
=&\displaystyle{\int_{t_1}^{t_2} f'(Y_t^\varepsilon)dY_t^\varepsilon+\dfrac{1}{2}\int_{t_1}^{t_2} f''(Y_t^\varepsilon)dt}
\\
=&\displaystyle{\int_{t_1}^{t_2} f'(Y_t^\varepsilon)\left(\dfrac{1}{\varepsilon}(X_t^\varepsilon)^2-Y_t^\varepsilon\right)dt+
\int_{t_1}^{t_2} f'(Y_t^\varepsilon)dW_t^2+
\dfrac{1}{2}\int_{t_1}^{t_2} f''(Y_t^\varepsilon)dt \ .}
\end{array}$$

This gives

\begin{equation}\label{Cor:HitClosenessToZeroSquareXOverEps:OutsideDelta:Eq:ItoIdentity}
\mathbf{E} \left[f(Y_{t_2}^\varepsilon)-f(Y_{t_1}^\varepsilon)-\int_{t_1}^{t_2} Af(Y_t^\varepsilon)dt\right]
=\mathbf{E}\int_{t_1}^{t_2} \left(\dfrac{1}{\varepsilon}(X_t^\varepsilon)^2-\dfrac{1}{2Y_t^\varepsilon}\right)f'(Y_t^\varepsilon)dt \ .
\end{equation}

From the proof of Lemma \ref{Lm:HitClosenessToZeroSquareXOverEps:OutsideDelta} we see that the estimate \eqref{Lm:HitClosenessToZeroSquareXOverEps:OutsideDelta:Eq:MartingaleErrorY-r} is valid also for the integral
from $t_1$ to $t_2$. In fact, a finer estimate can be obtained by improving
\eqref{Lm:HitClosenessToZeroSquareXOverEps:OutsideDelta:Eq:MartingaleErrorY-r-2} via the estimate
$$\mathbf{E}\int_{t_1}^{t_2} (X_t^\varepsilon)^2dt\leq C\varepsilon^{1-2\alpha}(t_2-t_1) \ .$$
So
$$\left|\mathbf{E}\left[f(Y_{t_2}^\varepsilon)-f(Y_{t_1}^\varepsilon)-\displaystyle{\int_{t_1}^{t_2}Af(Y_t^\varepsilon)dt}\right]\right|\leq
C[\varepsilon^{1-2\alpha}+\varepsilon^{1-4\alpha}(t_2-t_1)] \ .$$
Thus by \eqref{Cor:HitClosenessToZeroSquareXOverEps:OutsideDelta:Eq:ItoIdentity} we see that
$$\left|\mathbf{E}\int_{t_1}^{t_2} \left(\dfrac{1}{\varepsilon}(X_t^\varepsilon)^2-\dfrac{1}{2Y_t^\varepsilon}\right)f'(Y_t^\varepsilon)dt\right|
\leq C[\varepsilon^{1-2\alpha}+\varepsilon^{1-4\alpha}(t_2-t_1)] \ .$$
We can pick a function $f\in D(A)$ with bounded derivatives up to third order, such that $f'(y)\geq \delta^2$ for $y\geq \delta$. From here
we derive \eqref{Cor:HitClosenessToZeroSquareXOverEps:OutsideDelta:WeakEstimateSquareXOverEps}.
\end{proof}

\begin{lemma}\label{Lm:HitBackToZeroX:AbsXBiggerThanHalfKappa}
For any $\delta=\delta(\varepsilon)$ such that $\delta=\varepsilon^\alpha\rightarrow 0$ as $\varepsilon\downarrow 0$ for $\alpha=\dfrac{1}{10}$,
for any initial condition $|X_0^\varepsilon|\geq 2\delta$, the flow will quickly bring the particle
back to the region $Y\geq \delta$, and during this process the $|X|$--value is less or equal than $3\delta$.
In particular, this implies that $\mathbf{P}\left(|X_t^\varepsilon|\leq 3\delta \text{ for } 0\leq t \leq T\right)\rightarrow 1$ as $\varepsilon\downarrow 0$.
\end{lemma}

\begin{proof}
Let us introduce the angular variable $\theta_t^\varepsilon=\arctan\left(\dfrac{Y_t^\varepsilon}{X_t^\varepsilon}\right)$.
Here we take the principal branch of the function $\tan \theta$ as $\theta\in \left[-\dfrac{\pi}{2}, \dfrac{\pi}{2}\right]$.
Due to symmetry of the system \eqref{Eq:ABmodelPerturbedTimeRescaled} with respect to the $Y$--axis,
if the point $\theta\in \left[\dfrac{\pi}{2}, \dfrac{3\pi}{2}\right]$, then we can equivalently consider $\widetilde{\theta}=\pi-\theta$
as a replacement of $\theta$. In this way, if $\theta_t^\varepsilon=\dfrac{\pi}{2}$, then the diffusion particle is on the $Oy_A$ axis, and
if $\theta_t^\varepsilon=-\dfrac{\pi}{2}$, then the diffusion particle is on the $Oy_B$ axis. Let us apply It\^{o}'s formula
from \eqref{Eq:ABmodelPerturbedTimeRescaled} to $\theta_t^\varepsilon$ and we obtain
\begin{equation}\label{Lm:HitBackToZeroX:AbsXBiggerThanHalfKappa:Eq:AngularVariable}
\begin{array}{ll}
d\theta_t^\varepsilon
&=-\dfrac{Y_t^\varepsilon}{(X_t^\varepsilon)^2+(Y_t^\varepsilon)^2}\left(-\dfrac{1}{\varepsilon}X_t^\varepsilon Y_t^\varepsilon dt-X_t^\varepsilon dt+dW^1_t\right)
\\
& \qquad +\dfrac{X_t^\varepsilon}{(X_t^\varepsilon)^2+(Y_t^\varepsilon)^2}\left(\dfrac{1}{\varepsilon}(X_t^\varepsilon)^2dt-Y_t^\varepsilon dt+dW^2_t\right)
\\
& \qquad -\dfrac{1}{2}\left(\dfrac{2X_t^\varepsilon Y_t^\varepsilon}{[(X_t^\varepsilon)^2+(Y_t^\varepsilon)^2]^2}dt-\dfrac{2X_t^\varepsilon Y_t^\varepsilon}{[(X_t^\varepsilon)^2+(Y_t^\varepsilon)^2]^2}dt\right)
\\
&=\dfrac{1}{\varepsilon} X_t^\varepsilon dt +\dfrac{-Y_t^\varepsilon dW^1_t+X_t^\varepsilon dW^2_t}{(X_t^\varepsilon)^2+(Y_t^\varepsilon)^2}
\\
&=\dfrac{1}{\varepsilon} X_t^\varepsilon dt +\dfrac{1}{(r_t^\varepsilon)^2} dW_t^\theta \ .
\end{array}
\end{equation}
Here $W_t^\theta$ is another standard Brownian motion on $\mathbb{R}$. Comparing \eqref{Lm:HitBackToZeroX:AbsXBiggerThanHalfKappa:Eq:AngularVariable}
with \eqref{Eq:RadialProcessBessel} we see that we have the system

\begin{equation}\label{Lm:HitBackToZeroX:AbsXBiggerThanHalfKappa:Eq:Fast-SlowSystemRadial-Angular}
\left\{\begin{array}{l}
d\theta_t^\varepsilon=\dfrac{1}{\varepsilon} X_t^\varepsilon dt +\dfrac{1}{(r_t^\varepsilon)^2} dW_t^\theta \ , \theta_0^\varepsilon=\arctan\left(\dfrac{Y_0^\varepsilon}{X_0^\varepsilon}\right) \ ,
\\
dr_t^\varepsilon=\left(\dfrac{1}{2r_t^\varepsilon}-r_t^\varepsilon\right)dt+dW^r_t \ , r_0^\varepsilon=\sqrt{(X_0^\varepsilon)^2+(Y_0^\varepsilon)^2} \ .
\end{array}\right.
\end{equation}

The processes $W_t^\theta$ and $W_t^r$ are two driving standard Brownian motions on $\mathbb{R}$. Set the slow time clock
$\mathrm{t}=(\delta/\varepsilon)t$ and let us consider
a time--rescaled pair of processes $\Theta_{\mathrm{t}}^\varepsilon=\theta_{(\varepsilon/\delta)\mathrm{t}}^\varepsilon$ and
$R_\mathrm{t}^\varepsilon=r_{(\varepsilon/\delta)\mathrm{t}}^\varepsilon$.
Then the stochastic differential equations satisfied by $(\Theta_\mathrm{t}^\varepsilon, R_\mathrm{t}^\varepsilon)$
are given by

\begin{equation}\label{Lm:HitBackToZeroX:AbsXBiggerThanHalfKappa:Eq:Fast-SlowSystem-TimeRescaled-Radial-Angular}
\left\{\begin{array}{l}
d\Theta_\mathrm{t}^\varepsilon=\dfrac{X_{(\varepsilon/\delta) \mathrm{t}}^{\varepsilon}}{\delta}d\mathrm{t}+\sqrt{\dfrac{\varepsilon}{\delta}}\cdot
\dfrac{1}{(R_\mathrm{t}^\varepsilon)^2}dW^\theta_\mathrm{t} \ , \ \Theta_0^\varepsilon=\theta_0^\varepsilon
\\
dR_\mathrm{t}^\varepsilon=\dfrac{\varepsilon}{\delta}\left(\dfrac{1}{2R_\mathrm{t}^\varepsilon}-R_\mathrm{t}^\varepsilon\right)d\mathrm{t}+\sqrt{\dfrac{\varepsilon}{\delta}}dW^r_\mathrm{t} \ , R_0^\varepsilon=r_0^\varepsilon \ .
\end{array}\right.
\end{equation}

Without loss of generality,
let us start the process from some $(X_0^\varepsilon, Y_0^\varepsilon)$ such that $X_0^\varepsilon\geq 2\delta$ and $Y_0^\varepsilon\leq \delta$.
In this case we have $R_0^\varepsilon\geq 2\delta$. Consider the stopping time
\begin{equation}\label{Lm:HitBackToZeroX:AbsXBiggerThanHalfKappa:Eq:StoppingTimeXLessEqDt}
T_X^\varepsilon=\inf\{t\geq 0: X_t^\varepsilon\leq \delta\}
\end{equation}
and let $\mathrm{T}_X^\varepsilon=(\delta/\varepsilon)T_X^\varepsilon$. We see that for $t\in [0, T_X^\varepsilon]$ we have $X_t^\varepsilon\geq \delta$ and
thus $\dfrac{X_t^\varepsilon}{\delta}\geq 1$. We pick $\delta=\varepsilon^{\alpha}$ with $\alpha=\dfrac{1}{10}$. Since
$\sqrt{\dfrac{\varepsilon}{\delta}}=\varepsilon^{9/20}$ and $\dfrac{\varepsilon}{\delta^2}=\varepsilon^{4/5}=\varepsilon^{16/20}$, it is seen from the $R$--equation in
\eqref{Lm:HitBackToZeroX:AbsXBiggerThanHalfKappa:Eq:Fast-SlowSystem-TimeRescaled-Radial-Angular} that for finite
$\mathrm{t}$,
\begin{equation}\label{Lm:HitBackToZeroX:AbsXBiggerThanHalfKappa:Eq:NoChangeInRadialProcess}
\mathbf{P}\left(|R^\varepsilon_{\mathrm{t}}-2\delta|\leq C\varepsilon^{9/20}\right)=1 \ .
\end{equation}
In this case, $\sqrt{\dfrac{\varepsilon}{\delta}}\cdot\dfrac{1}{(R_{\mathrm{t}}^\varepsilon)^2}\sim \mathcal{O}\left(\dfrac{\varepsilon^{1/2}}{\delta^2}\right)
=\mathcal{O}(\varepsilon^{1/2-1/5})=\mathcal{O}(\varepsilon^{3/10})$. Therefore, the $\Theta$--equation in
\eqref{Lm:HitBackToZeroX:AbsXBiggerThanHalfKappa:Eq:Fast-SlowSystem-TimeRescaled-Radial-Angular}
can be viewed as a perturbation of the dynamical equation

\begin{equation}\label{Lm:HitBackToZeroX:AbsXBiggerThanHalfKappa:Eq:Fast-SlowSystem-TimeRescaled-Angular-NoPerturbation}
d\Theta_\mathrm{t}=\dfrac{X_{(\varepsilon/\delta) \mathrm{t}}^{\varepsilon}}{\delta}d\mathrm{t} \ , \ \Theta_0=\theta_0^\varepsilon \ ,
\end{equation}
such that
\begin{equation}\label{Lm:HitBackToZeroX:AbsXBiggerThanHalfKappa:Eq:ClosenessAngularProcessZeroPerturbation}
\mathbf{P}\left(|\Theta_{\mathrm{t}}^\varepsilon-\Theta_{\mathrm{t}}|\leq C \varepsilon^{3/10}\right)=1
\end{equation}
in finite $\mathrm{t}$. From \eqref{Lm:HitBackToZeroX:AbsXBiggerThanHalfKappa:Eq:StoppingTimeXLessEqDt},
\eqref{Lm:HitBackToZeroX:AbsXBiggerThanHalfKappa:Eq:NoChangeInRadialProcess},
 \eqref{Lm:HitBackToZeroX:AbsXBiggerThanHalfKappa:Eq:Fast-SlowSystem-TimeRescaled-Angular-NoPerturbation} and \eqref{Lm:HitBackToZeroX:AbsXBiggerThanHalfKappa:Eq:ClosenessAngularProcessZeroPerturbation}
we know that $\mathrm{T}_X^\varepsilon$ is finite, and thus $T_X^\varepsilon\sim \mathcal{O}(\varepsilon^{9/10})$ and $Y_{T_X^\varepsilon}^\varepsilon\geq \delta$.
From here we know that whenever $X_0^\varepsilon\geq 2\delta$, the flow will quickly bring the particle to the region $Y\geq \delta$,
and during this process $X_t^\varepsilon\leq 3\delta$. Thus we see that with high probability, we have $X_t^\varepsilon\leq 3\delta$. The other--side
estimate $X_t^\varepsilon\geq -3\delta$ is obtained in a same fashion.
\end{proof}

\begin{lemma}\label{Lm:HitBackToZeroX:YLessThanMinusDelta}
For any $\delta=\delta(\varepsilon)$ such that $\delta=\varepsilon^\alpha\rightarrow 0$ as $\varepsilon\downarrow 0$ for $\alpha=\dfrac{1}{10}$,
for any initial condition $Y_0^\varepsilon\leq -1.5\delta$, the flow will quickly bring the particle
 back to the region $Y\geq \delta$, and during this process the $Y$--coordinate is $\geq -1.99\delta$
 with probability $\rightarrow 1$ as $\varepsilon\downarrow 0$.
\end{lemma}

\begin{proof}
This is proved in the same way as the proof for Lemma \ref{Lm:HitBackToZeroX:AbsXBiggerThanHalfKappa}.
\end{proof}

\begin{lemma}\label{Lm:ExitTimeEstimate:NearO:TauToSigma}
We have $\mathbf{E}(\sigma_k-\tau_k)\leq C\delta^2\rightarrow 0$ as $\varepsilon\downarrow 0$ for some constant $C>0$.
\end{lemma}

\begin{proof}
Let us introduce the auxiliary OU--process
\begin{equation}\label{Lm:ExitTimeEstimate:NearO:TauToSigma:Eq:AuxiliaryOU-Yprocess}
d\widehat{Y}_t=-\widehat{Y}_t dt+dW_t^2 \ , \ \widehat{Y}_0=Y_0^\varepsilon \ .
\end{equation}

By Lemma \ref{Lm:HitBackToZeroX:YLessThanMinusDelta}, we know that as $\varepsilon$ is small, with probability close to $1$
we have $Y^\varepsilon_{\tau_k}=2\delta$. Taking this into account, as we have $\dfrac{(X_t^\varepsilon)^2}{\varepsilon}\geq 0$,
we can estimate by comparison that
$$\mathbf{E}(\sigma_k-\tau_k)\leq \mathbf{E} \left(\sigma|\widehat{Y}_{\sigma}=2\delta\right) \ .$$

Here $\sigma$ is the first time that the OU--process $\widehat{Y}_t$
starting from $\widehat{Y}_0=\delta$ hits $Y=\pm 2\delta$. As we have
$$\begin{array}{ll}
\mathbf{E}\sigma & =\mathbf{E}\left(\sigma|\widehat{Y}_{\sigma}=2\delta\right)\mathbf{P}(\widehat{Y}_\sigma=2\delta)+
\mathbf{E}\left(\sigma|\widehat{Y}_{\sigma}=-2\delta\right)\mathbf{P}(\widehat{Y}_\sigma=-2\delta)
\\
&\geq \mathbf{E}\left(\sigma|\widehat{Y}_{\sigma}=2\delta\right)\mathbf{P}(\widehat{Y}_\sigma=2\delta)
\\
& = \dfrac{3}{4}\mathbf{E}\left(\sigma|\widehat{Y}_{\sigma}=2\delta\right) \ ,
\end{array}$$
we can further estimate
\begin{equation}\label{Lm:ExitTimeEstimate:NearO:TauToSigma:Eq:EstimateExitTimeInTermsOU}
\mathbf{E}(\sigma_k-\tau_k)\leq \dfrac{4}{3}\mathbf{E}\sigma \ .
\end{equation}

We denote $u(\delta)=\mathbf{E}\sigma$. By the standard theory of
stochastic differential equations we know that $u(y)$, $y\in [-2\delta, 2\delta]$ is the solution to the ODE
$$\left\{\begin{array}{l}
-yu'(y)+\dfrac{1}{2}u''(y)=-1 \ ,
\\
u(2\delta)=u(-2\delta)=0 \ .
\end{array}\right.$$
Solving the above ODE system, we obtain that
$$u(y)=-2\int_{-2\delta}^y e^{z^2}dz\int_{-2\delta}^z e^{-u^2}du+2\int_{-2\delta}^y e^{z^2}dz
\dfrac{\displaystyle{\int_{-2\delta}^{2\delta}e^{z^2}dz\int_{-2\delta}^z e^{-u^2}du}}
{\displaystyle{\int_{-2\delta}^{2\delta}e^{z^2}dz}} \ .$$
It is easy to see that as $y\in [-2\delta, 2\delta]$ we have
$0\leq \dfrac{\displaystyle{\int_{-2\delta}^y e^{z^2}dz}}{\displaystyle{\int_{-2\delta}^{2\delta}e^{z^2}dz}}\leq 1$. Thus
$$0\leq u(\delta)\leq \displaystyle{\int_{\delta}^{2\delta}e^{z^2}dz\int_{-2\delta}^z e^{-u^2}du} \ .$$
In particular, this implies that
$u(\delta)\leq C\delta^2$ for some $C>0$. Taking into account \eqref{Lm:ExitTimeEstimate:NearO:TauToSigma:Eq:EstimateExitTimeInTermsOU},
we obtain the statement of this Lemma.
\end{proof}

\begin{lemma}\label{Lm:ExitTimeEstimate:FarFromO:SigmaToTau}
We have $\mathbf{E}(\tau_{k+1}-\sigma_k)\geq C\delta$ as $\varepsilon\downarrow 0$ for some constant $C>0$.
\end{lemma}

\begin{proof}
Recall that the $Y$--equation in \eqref{Eq:ABmodelPerturbedTimeRescaled} has the form
$$dY_t^\varepsilon=\left(\dfrac{1}{\varepsilon}(X_t^\varepsilon)^2-Y_t^\varepsilon\right)dt+dW_t^2 \ .$$
Thus by comparison, we know that
$$Y_t^\varepsilon\geq \widehat{Y}_t \ ,$$
in which $\widehat{Y}_t$ is an OU--process defined by
$$d\widehat{Y}_t=-\widehat{Y}_tdt+dW_t^2 \ , \widehat{Y}_0=Y_0^\varepsilon \ .$$
From here, we know that we have
$$\mathbf{E}(\tau_{k+1}-\sigma_k)\geq \mathbf{E} \tau \ ,$$
where $\tau$ is the first time that the process $\widehat{Y}_t$ starting from $2\delta$
hits $\delta$.

Set $u(2\delta)=\mathbf{E}\tau$. From the standard theory of stochastic differential equations we infer that $u(y)$, $y\in [\delta,\infty)$
is the solution to the ODE
$$\left\{\begin{array}{l}
-yu'(y)+\dfrac{1}{2}u''(y)=-1 \ ,
\\
u(\delta)=u(\infty)=0 \ .
\end{array}\right.$$
Solving the above ODE system, we obtain, for $y\in [\delta, \infty)$, that $u(y)=\lim\limits_{M\rightarrow\infty} u_M(y)$, where
$$u_M(y)=-2\int_{M}^y e^{z^2}dz\int_{M}^z e^{-u^2}du+2\int_{M}^y e^{z^2}dz
\dfrac{\displaystyle{\int_{M}^{\delta}e^{z^2}dz\int_{M}^z e^{-u^2}du}}
{\displaystyle{\int_{M}^{\delta}e^{z^2}dz}} \ .$$
Again, as $M\rightarrow\infty$ we have $\lim\limits_{M\rightarrow\infty}\dfrac{\displaystyle{\int_{M}^y e^{z^2}dz}}{\displaystyle{\int_{M}^{\delta}e^{z^2}dz}}=1$. Thus
in the limit we have
$$\mathbf{E}\tau=u(2\delta)=2\int_{\delta}^{2\delta}e^{z^2}dz\int_{z}^{\infty} e^{-u^2}du\geq 2\delta e^{\delta^2}\int_{2\delta}^\infty e^{-u^2}du\geq C\delta$$
for some constant $C>0$.
\end{proof}

\begin{lemma}\label{Lm:UpperBoundNumberOfCrossingsBeforeT}
The number of up--crossings $N(\varepsilon)$ from $\delta$ to $2\delta$ before time $T$ has the asymptotic $N(\varepsilon)\leq CT\delta^{-1}$
for some constant $C>0$.
\end{lemma}

\begin{proof}
This follows from Lemma \ref{Lm:ExitTimeEstimate:FarFromO:SigmaToTau}.
\end{proof}

\begin{lemma}\label{Lm:TightnessYProcess}
The process $Y_t^\varepsilon$ is weakly compact in $\mathbf{C}_{[0,T]}(\mathbb{R})$.
\end{lemma}

\begin{proof}
Let $(\Omega, \mathcal{F}, \mathbf{P})$ be the probability space for $Y_t^\varepsilon$, $0\leq t\leq T$, such that
for any $\omega\in \Omega$ the sample path $Y^\varepsilon_t(\omega)$, $0\leq t\leq T$ is a trajectory in $\mathbf{C}_{[0,T]}(\mathbb{R})$.
We would like to show that from any sequence $\varepsilon_k\downarrow 0$, $k=1,2,...$ as $k\rightarrow \infty$
one can extract a further subsequence $\varepsilon_{k_j}\downarrow 0$, $j=1,2,...$ as $j\rightarrow \infty$ such that
for any bounded continuous functional $F$ on $\mathbf{C}_{[0,T]}(\mathbb{R})$
we have
\begin{equation}\label{Lm:TightnessYProcess:WeakConvergenceY}
\mathbf{E} F(Y_t^{\varepsilon_{k_j}}(\omega))\rightarrow \mathbf{E} F(Y_t^0(\omega))
\end{equation}
for some $j\rightarrow \infty$ and some random element $Y_t^0$ in $\mathbf{C}_{[0,T]}(\mathbb{R})$. Here $\mathbf{E}$ is the expectation
with respect to $\mathbf{P}$.

Unlike any of the previous Lemmas, here we will pick some \textit{fixed} $\delta>0$.
It is easy to see that if we replace $\delta=\varepsilon^{\alpha}$ by a fixed $\delta$, then
Lemmas \ref{Lm:HitBackToZeroX:YLessThanMinusDelta}, \ref{Lm:ExitTimeEstimate:NearO:TauToSigma},
\ref{Lm:ExitTimeEstimate:FarFromO:SigmaToTau} remain valid
(The stopping times $\sigma_k$ and $\tau_k$ can also be defined in a same way as for $\delta=\varepsilon^{\alpha}$),
while the estimate
\eqref{Lm:HitClosenessToZeroX:OutsideDelta:Eq:EstimateOfSquareX} in Lemma
\ref{Lm:HitClosenessToZeroX:OutsideDelta} shall be modified into

\begin{equation}\label{Lm:TightnessYProcess:HitClosenessToZeroX:OutsideDelta}
\mathbf{E} (X_t^\varepsilon)^2\leq C\dfrac{\varepsilon}{\delta} \ .
\end{equation}

Henceforce we will make use of Lemmas \ref{Lm:HitBackToZeroX:YLessThanMinusDelta}, \ref{Lm:ExitTimeEstimate:NearO:TauToSigma},
\ref{Lm:ExitTimeEstimate:FarFromO:SigmaToTau} in below by directly adapting it to a fixed $\delta>0$.

Let for any small $\varepsilon>0$ the family of sample pathes
\begin{equation}\label{Lm:TightnessYProcess:BadEvent}
\Omega_{\text{bad}}^{\varepsilon,\delta}=\{\omega: \min\limits_{0\leq t\leq T}Y_t^\varepsilon(\omega)\leq -2\delta\} \ .
\end{equation}
By Lemma \ref{Lm:HitBackToZeroX:YLessThanMinusDelta} we know that
$\mathbf{P}(\Omega_{\text{bad}}^{\varepsilon,\delta})\rightarrow 0$ as $\varepsilon\downarrow 0$.

Let us introduce a new probability measure $\widehat{\mathbf{P}}$ on $(\Omega, \mathcal{F}, \mathbf{P})$
as follows. For any event $A\in \mathcal{F}$ we define
\begin{equation}
\widehat{\mathbf{P}}(A)=\dfrac{\mathbf{P}(A\backslash \Omega_{\text{bad}}^{\varepsilon,\delta})}
{\mathbf{P}(\Omega\backslash \Omega_{\text{bad}}^{\varepsilon,\delta})} \ .
\end{equation}
Let the corresponding expectation be defined by $\widehat{\mathbf{E}}$. As we have
$\mathbf{P}(\Omega_{\text{bad}}^{\varepsilon,\delta})\rightarrow 0$ as $\varepsilon\downarrow 0$, we have that $\widehat{\mathbf{E}} X\rightarrow \mathbf{E} X$ for
any random variable $X$ as $\varepsilon \downarrow 0$.
From here we see that to show \eqref{Lm:TightnessYProcess:WeakConvergenceY}
it suffices to show that
\begin{equation}\label{Lm:TightnessYProcess:WeakConvergenceYhat}
\widehat{\mathbf{E}} F(Y_t^{\varepsilon_{k_j}}(\omega))\rightarrow \widehat{\mathbf{E}} F(Y_t^0(\omega))
\end{equation}
for some $j\rightarrow \infty$ and some random element $Y_t^0$ in $\mathbf{C}_{[0,T]}(\mathbb{R})$.
We then understand \eqref{Lm:TightnessYProcess:WeakConvergenceYhat} is just saying
that $Y_t^\varepsilon$ is weakly--compact under $\widehat{\mathbf{P}}$.
We will then make use of Lemma 5.1 in \cite{[FWfishpaper]}. In fact, Lemma 5.1 in \cite{[FWfishpaper]}
indicates that in order to show weak--compactness of the family of sample paths in $Y^\varepsilon_t$ in $\mathbf{C}_{[0,t]}(\mathbb{R})$
under the measure $\widehat{\mathbf{P}}$,
it suffices to show, for each $\delta>0$, weak--compactness of the family of sample paths
$\widetilde{Y}_t^{\varepsilon,\delta}$, where $\widetilde{Y}_t^{\varepsilon,\delta}=Y_t^{\varepsilon}$
for $\sigma_{k-1}\leq t\leq \tau_k$, $k=1,2,...,N$ and
$$\widetilde{Y}_t^{\varepsilon,\delta}=\delta\dfrac{\tau_k-t}{\tau_k-\sigma_k}+2\delta\dfrac{t-\sigma_k}{\tau_k-\sigma_k}$$
for $\tau_k\leq t\leq \sigma_k$. This is because we have $|Y_t^\varepsilon(\omega)-\widetilde{Y}_t^{\varepsilon,\delta}(\omega)|\leq 4\delta$
for each $\delta>0$ on $\omega\in \Omega\backslash \Omega_{\text{bad}}^{\varepsilon,\delta}$.

By the classical Prokhorov's theorem, to show weak--compactness of the process $\widetilde{Y}_t^{\varepsilon,\delta}$,
it suffices to check tightness of the family of processes $\widetilde{Y}_t^{\varepsilon,\delta}$, $0\leq t \leq T$.
Since $\widetilde{Y}_t^{\varepsilon,\delta}$ is a linear interpolation between $\tau_k\leq t\leq \sigma_k$, we just have to
check that, for any $\sigma_{k-1}\leq s_1\leq s_2\leq \tau_k$ so that $|s_2-s_1|$ is small,
\begin{equation}\label{Lm:TightnessYProcess:widetildeYProcessEquiContinuity}
\widehat{\mathbf{E}}|\widetilde{Y}_{s_2}^{\varepsilon,\delta}-\widetilde{Y}_{s_1}^{\varepsilon,\delta}|^{a}\leq C|s_1-s_2|^{1+b} \ ,
\end{equation}
for some $a,b>0$ and $C>0$. Since $\widetilde{Y}_{s}^{\varepsilon,\delta}=Y_{s}^\varepsilon$ for $\sigma_{k-1}\leq s \leq \tau_k$,
and $\mathbf{P}(\Omega_{\text{bad}}^{\varepsilon,\delta})\rightarrow 0$ as $\varepsilon\downarrow 0$, we just have to check
\eqref{Lm:TightnessYProcess:YProcessEquiContinuity} for $Y_s^\varepsilon$ and $\widehat{\mathbf{E}}$ replaced by $\mathbf{E}$, i.e.
\begin{equation}\label{Lm:TightnessYProcess:YProcessEquiContinuity}
\mathbf{E}|Y_{s_2}^{\varepsilon}-Y_{s_1}^{\varepsilon}|^{a}\leq C|s_1-s_2|^{1+b} \ .
\end{equation}
Notice that, for any
$\sigma_{k-1}\leq s_1\leq s_2\leq \tau_k$, we have
\begin{equation}\label{Lm:TightnessYProcess:YProcessIntegralEquation}
Y_{s_2}^\varepsilon-Y_{s_1}^\varepsilon=\dfrac{1}{\varepsilon}\int_{s_1}^{s_2} (X_s^\varepsilon)^2ds-\int_{s_1}^{s_2}Y_s^\varepsilon ds+(W_{s_2}^2-W_{s_1}^2) \ .
\end{equation}
From here, we see that \eqref{Lm:TightnessYProcess:YProcessEquiContinuity}
follows from \eqref{Lm:TightnessYProcess:HitClosenessToZeroX:OutsideDelta}.

\end{proof}

\section{``Metastable" behavior of the system as $\varepsilon\downarrow 0$.}\label{Sec:Metastable}

The previous section considered the case when $\varepsilon\downarrow 0$. In this case, one can roughly understand that the coupled
process $(X_t^\varepsilon, Y_t^\varepsilon)$ converges weakly to $(0, Y_t)$. We can then let $t\rightarrow \infty$, so that the damped--BES(2) process
$Y_t$ in \eqref{Eq:RadialProcessBessel} converges to an invariant measure $\mu^Y$ on $Oy_A$. In this case, ignoring the topology with respect to which we speak
about convergence, one can say very vaguely that
$$\lim\limits_{t\rightarrow\infty}\lim\limits_{\varepsilon\downarrow 0}\  (X_t^\varepsilon, Y_t^\varepsilon)=(0 \ , \ \mu^Y \text{ on } Oy_A) \ .$$
It is in this sense that we can understand the measure $\mu^Y$ on $Oy_A$ as a global ``attractor" of
our system $(X_t^\varepsilon, Y_t^\varepsilon)$. One can also consider the case when the two limits are inverted, namely
for any given measurable set $\Gamma\subseteq\mathbb{R}^2$ we have the convergence of the form
$$\lim\limits_{\varepsilon\downarrow 0}\lim\limits_{t\rightarrow\infty}\mathbf{P}\left((X_t^\varepsilon, Y_t^\varepsilon)\in \Gamma\right)=\mu_0(\Gamma) \ .$$
The limiting measure $\mu_0(\Gamma)$ has been studied in \cite{[BT]} via invariant measure and Kolmogorov (Fokker--Plank)
equation, and has been shown to concentrate
on $Oy_A$. In the classical theory regarding random perturbations of dynamical systems
(see \cite[Section 6.6]{[FWbook2012]}), one is interested in considering the above two limits in a coordinated way. Namely we consider
the case when $t=t(\varepsilon)\rightarrow \infty$ as $\varepsilon\downarrow 0$, and the asymptotic distribution of $(X_{t(\varepsilon)}^\varepsilon, Y_{t(\varepsilon)}^\varepsilon)$.
In the classical case such as those demonstrated in \cite{[FWbook]}, \cite{[FWbook2012]},
the $\omega$--limit sets of the unperturbed system consists of isolated compactum.
In this case, if $t(\varepsilon)$ increases sufficiently slowly, then over time $t(\varepsilon)$ the trajectory of $(X_{t(\varepsilon)}^\varepsilon, Y_{t(\varepsilon)}^\varepsilon)$
cannot move far from that stable compactum in whose domain of attraction the initial
point is. Over larger time intervals there are passages from the neighborhood
of this compactum to neighborhoods of others: first to the ``closest" compactum
(in the sense of the action functional) and then to more and more ``far away"
ones. Such a phenomenon has been quantitatively characterized as the ``metastable" behavior of the system.

The particular feature of
the system \eqref{Intro:Eq:ABmodelPerturbed} that we consider here has
been in that the unperturbed system admits a continuum of stable attractors. At the level of
time--rescaled process \eqref{Eq:ABmodelPerturbedTimeRescaled}, this leads to possible
``jumps" of $(X_{t(\varepsilon)}^\varepsilon, Y_{t(\varepsilon)}^\varepsilon)$ between $Oy_A$ and $Oy_B$. To illustrate this, let us imagine that
we start our process $(X_t^\varepsilon, Y_t^\varepsilon)$ in \eqref{Eq:ABmodelPerturbedTimeRescaled}
from $(X_0^\varepsilon, Y_0^\varepsilon)$ such that $Y_0^\varepsilon\geq 0$.

As $\varepsilon$ is small, in very short time $\sim \mathcal{O}(\varepsilon)$, the process $(X_t^\varepsilon, Y_t^\varepsilon)$ first
comes close to the $Y$--axis along the deterministic flow, and it hits a neighborhood of
$(0, y^\pi(X_0^\varepsilon, Y_0^\varepsilon))$ \footnote{Recall the definition of $y^\pi(x_0, y_0)$ in Definition \ref{Def:ProjectionOperatorPi}.}.
For any $a>0$, let the stopping time
\begin{equation}\label{Eq:StppingTimeYleq-a}
T(a;\varepsilon)=\inf\{t\geq 0; Y_t^\varepsilon\leq -a\} \ .
\end{equation}
We then define
\begin{equation}\label{Eq:ProbabilityYleq-a}
p(a, t;\varepsilon)=\mathbf{P}_{(X_0^\varepsilon, Y_0^\varepsilon)}\left(T(a;\varepsilon)\leq t\right)
\end{equation}
to be the probability that the trajectory $\{(X_s^\varepsilon, Y_s^\varepsilon)\}_{0\leq s \leq t}$ ever reached below $Y=-a$ on the $Oy_B$ axis.
By Lemma \ref{Lm:HitBackToZeroX:YLessThanMinusDelta}, we have
that $p(a, t; \varepsilon)\rightarrow 0$ as $\varepsilon\downarrow 0$. We set $t(\varepsilon)\sim \dfrac{1}{p(a,t;\varepsilon)} \rightarrow \infty$
as $\varepsilon \downarrow 0$. Then we see that at time scale $\sim t(\varepsilon)$ the process $(X_{t(\varepsilon)}^\varepsilon, Y_{t(\varepsilon)}^\varepsilon)$
may demonstrate an excursion to $Y\leq -a$. By combining Lemma \ref{Lm:HitClosenessToZeroX:OutsideDelta}
 and the instability of the flow near $Oy_B$, we see that this excursion happens along the $Y$--axis
 and will hit in a neighborhood of $(0, -a)$.
In fact, within the half space for $Y>0$ the process $(X_t^\varepsilon, Y_t^\varepsilon)$ will be pushed by the deterministic flow
to be close to the $Y$--axis.
When the excursion diffuses to the half--space with $Y<0$ but $|X|\neq 0$, the deterministic flow will
quickly bring the process $(X_t^\varepsilon, Y_t^\varepsilon)$ back to the half--space with positive $Y$--value.
Therefore the excursion to $(0,-a)$ within the half--space for $Y<0$ should happen along the $Y$--axis.
At time $t\sim t(\varepsilon)$, the process $(X_t^\varepsilon, Y_t^\varepsilon)$
will be close to $(0,-a)$ and is fluctuating in a neighborhood of this point.
Due to instability of the flow near $Oy_B$ axis, the process $(X_t^\varepsilon, Y_t^\varepsilon)$ will then be
quickly (at time scale $\sim \mathcal{O}(\varepsilon)$) brought back to a neighborhood of $(0, a)$.

Under the above mechanism, as $\varepsilon>0$ is small, what we actually see is that the process $(X_t^\varepsilon, Y_t^\varepsilon)$, although mostly
stays within the half--plane of positive $Y$--value, being close to the $Y$--axis, makes rare excursions
to $(0, -a)$ along $Y$--axis, and after that quickly jumps back to $(0,a)$.
As $\varepsilon>0$ becomes smaller and smaller, the excursion to $(0, -a)$ becomes rarer and rarer, so that in the limit $\varepsilon\downarrow 0$, the process
$(X_t^\varepsilon, Y_t^\varepsilon)$ will not enter $Oy_B$ any more, and we arrive at the ``process level stable attractor" $(0,Y_t)$. This characterizes
the metastable behavior of the system \eqref{Eq:ABmodelPerturbedTimeRescaled}, and when changed back to the slow time, the perturbed system
\eqref{Intro:Eq:ABmodelPerturbed}.

\section{Formulation of the system \eqref{Intro:Eq:ABmodel}
as the Euler--Arnold equation for the group of all affine transformations of a line.}\label{Sec:Euler-Arnold}

In a beautiful paper from 1966
(see \cite{[Arnold1966]}, also \cite[Appendix 2]{[Arnold]} and \cite{[TaoBlog]}), V.I.Arnold observed
that many basic equations in physics, including the Euler equations of the motion
of a rigid body, and also the Euler equations describing the fluid dynamics of an inviscid incompressible fluid,
can be viewed (formally, at least) as geodesic flows on a (finite or infinite dimensional) Riemannian manifold $G$.
This Riemannian manifold $G$ is also a Lie group equipped with a right--invariant metric.
Equivalently, these geodesic flows can be written as the solution of an ordinary differential equation
at the co--tangent space to the origin of the Lie group $G$ (the dual space of the Lie algebra of $G$), describing the evolution of the angular momentum
(more precisely, the pull back of the angular momentum to the origin).
Such an ordinary differential equation has been thereafter named \textit{the Euler--Arnold equation}.
Below let us first briefly discuss the background of the Euler--Arnold equation in one subsection,
and then in another subsection we will formulate our system \eqref{Intro:Eq:ABmodel} as the
Euler--Arnold equation for the group of all affine transformations of a line.

\subsection{Background of the Euler--Arnold equation.}

Let $G$ be an $n$--dimensional real Lie group. Let $\mathfrak{g}$ be its Lie
algebra, i.e., the tangent space of $G$ at the identity element $e$
associated with a commutator relation $[ , ]$. The commutator
relation is defined in the standard way: For two tangent vectors
$\xi$ and $\eta$ the Lie bracket is defined as
$[\xi,\eta]=\displaystyle{\left.\dfrac{\partial^2}{\partial s \partial
t}\right|_{s=t=0}e^{t\xi}e^{s\eta}}e^{-t\xi}e^{-s\eta}$.
In a coordinate dependent language if $e_1,...,e_n$ be a basis of
$\mathfrak{g}$ so that $c_{ij}^k$ are structure constants, then $[e_i,
e_j]=\sum\limits_{k=1}^n c_{ij}^ke_k$.

Consider the actions of left and right shifts of $G$ on itself:

$$L_g: G\rightarrow G \ , \ L_g h=gh \ ; R_g: G \rightarrow G \ , \ R_gh=hg \ .$$

The induced maps on the tangent space at every $h\in G$ are

$$L_{g*}: T_hG \rightarrow T_{gh}G \ , \ R_{g*}: T_hG\rightarrow T_{hg}G \ .$$

Consider the diffeomorphism $R_{g^{-1}}L_g$, which is an inner
automorphism of the group $G$. This diffeomorphism preserves the
identity element $e$ and its derivative at the identity element $e$
is the so called \textit{adjoint representation} $Ad_g$ of the group
$G$. That is to say, $$Ad_g: \mathfrak{g}\rightarrow \mathfrak{g} \ , \
Ad_g=(R_{g^{-1}}L_g)_{*e} \ .$$

The mapping $Ad_g$ satisfies $Ad_g[\xi, \eta]=[Ad_g\xi, Ad_g\eta] \
, \ \xi, \eta\in \mathfrak{g}$ as well as $Ad_{gh}=Ad_gAd_h$. We could view
the mapping $Ad$ as a mapping from the group to the space of linear
operators on $\mathfrak{g}$: $$Ad(g)=Ad_g \ .$$ The derivative of the mapping
$Ad$ at the identity element $e$ of the group $G$ is a linear
mapping $ad$ from $\mathfrak{g}$ to the space of linear operators on $\mathfrak{g}$.
We have

$$ad=Ad_{*e}: \mathfrak{g} \rightarrow \text{End}\mathfrak{g} \ , \ ad_\xi=\left.\dfrac{d}{dt}\right|_{t=0}Ad_{e^{t\xi}} \ .$$

We see that
$ad_\xi\eta=\left.\dfrac{d}{dt}\right|_{t=0}Ad_{e^{t\xi}}\eta
=\left.\dfrac{d}{dt}\right|_{t=0}(R_{e^{-t\xi}}L_{e^{t\xi}})_{*e}\eta=
\left.\dfrac{\partial^2}{\partial t\partial
s}\right|_{t=s=0}e^{t\xi}e^{s\eta}e^{-t\xi}e^{-s\eta}=[\xi,\eta]$.

Let us now consider the dual space $\mathfrak{g}^*$ of the Lie algebra $\mathfrak{g}$.
The space $\mathfrak{g}^*$ consists of all real linear functionals on $\mathfrak{g}$:
$\mathfrak{g}^*=T_e^*G$. Let us denote the pairing of $\xi\in T_g^*G$ and
$\eta\in T_gG$ in the cotangent/tangent spaces at $g\in G$ by the
bracket
$$(\xi,\eta)\in \mathbb{R} \ , \xi\in T_g^*G \ , \eta\in T_gG \ .$$ It is
natural to define the dual operator $Ad_g^*: \mathfrak{g}^*\rightarrow \mathfrak{g}^*$ by the
identity $$(Ad_g^*\xi, \eta)=(\xi, Ad_g\eta) \ .$$
The operator $Ad_g^*$ is the \textit{co-adjoint representation} of
the group $G$.

Correspondingly, one can define
$$ad^*_\xi: \mathfrak{g}^*\rightarrow\mathfrak{g}^* \ , \xi\in
\mathfrak{g} \ , \ ad_\xi^*=\left.\dfrac{d}{dt}\right|_{t=0}Ad^*_{e^{t\xi}} \
,$$ such that
$$(ad^*_\xi\eta, \zeta)=(\eta, ad_\xi\zeta) \ , \eta\in \mathfrak{g}^* \ ,
\zeta\in \mathfrak{g} \ , \xi\in \mathfrak{g} \ .$$ We may denote
$$ad_\xi^*\eta=\{\xi, \eta\} \ , \ \xi\in \mathfrak{g} \ , \ \eta\in \mathfrak{g}^* \
.$$

We have an identity
$$(\{\xi,\eta\}, \zeta)=(\eta, [\xi, \zeta]) \ \text{ for } \ \xi\in \mathfrak{g} \ , \ \eta\in \mathfrak{g}^* \ , \ \zeta\in \mathfrak{g} \ .$$

Let us turn to coordinate--dependent language. If $e^1, ...,e^n$ is
a basis dual to $e_1,...,e_n$ in $\mathfrak{g}^*$: $(e^i, e_j)=\delta_j^i$. Then
we can calculate $\{e_i, e^j\}=\sum\limits_{k=1}^n c_{ik}^je^k$.

Let $A: \mathfrak{g}\rightarrow \mathfrak{g}^*$ be a symmetric and positive definite linear
operator: for any $\xi, \eta\in \mathfrak{g}$ we have $(A\xi, \xi)>0$ and
$(A\xi, \eta)=(A\eta, \xi)$. Let $A_g: T_gG\rightarrow T_g^*G$ be defined by
$A_g\xi=L_{g}^*AL_{g^{-1}*}\xi$, $\xi\in T_gG$. In mechanical
applications the operator $A_g$ gives the \textit{moment of
inertia}. Consider a metric on $G$ defined by an inner product
$$\langle\xi, \eta\rangle_g=(A_g\xi, \eta)=(A_g\eta, \xi)=\langle\eta,
\xi\rangle_g$$ for $\xi, \eta\in T_gG$. This metric is a
\textit{left invariant metric} on $G$, i.e., $\langle\xi,
\eta\rangle_e=\langle L_{g*}\xi, L_{g*}\eta\rangle_g$, and it makes the Lie group $G$ into a Riemannian manifold.
We shall
denote the corresponding inner product $\langle , \rangle_e$ at
$T_eG=\mathfrak{g}$ simply as $\langle , \rangle$. We shall also denote the
operator $A_e$ simply as $A$.

The above introduced inner product also induces an inner product on
$T_e^*G$. Let $\zeta\in T_e^*G$ and $\mu\in T_e^*G$. We can define
$\langle\zeta, \mu\rangle=\langle\zeta,\mu\rangle_e=(\zeta, A^{-1}\mu)$.
Such an inner product on $T_e^*G$ makes $T_e^*G$ into an inner
product space.

Consider a geodesic curve $g=g(t)$ on the group $G$, with respect to
the metric given by $\langle , \rangle_g$. The trajectory $g=g(t)$
complies with \textit{the principle of least action}. The Lagrangian
here is the \textit{kinetic energy} $T(t)=E(t)
=\dfrac{1}{2}\langle\dot{g}(t), \dot{g}(t)\rangle_{g(t)}$ and the
\textit{action} is
$S_{0t}(g)=\displaystyle{\int_0^t\dfrac{1}{2}\langle\dot{g}(s),
\dot{g}(s)\rangle_{g(s)}ds}$. The trajectory $g=g(t)$ is such that
the first variation of the action vanishes.

The \textit{angular velocity} is $\omega=\dot{g}$. Let

$$\omega_c=L_{g^{-1}*}\dot{g}\in \mathfrak{g} \ , \ \omega_s=R_{g^{-1}*}\dot{g}\in \mathfrak{g} \ .$$

These are the so called ``angular velocity in the body" ($\omega_c$)
and ``angular velocity in the space" ($\omega_s$).

The \textit{angular momentum} is defined as $$M=A_g \dot{g} \ .$$ We
see that $M\in T_g^*G$. We consider

$$M_c=L_{g^{-1}}^*M\in \mathfrak{g}^*  \ , \ M_s=R_{g^{-1}}^*M\in \mathfrak{g}^* \ .$$

These can be viewed as ``angular momentum in the body" ($M_c$) and
``angular momentum in the space" ($M_s$).

The kinetic energy can be rewritten as

$$T=E
=\dfrac{1}{2}\langle\dot{g}, \dot{g}\rangle_g
=\dfrac{1}{2}\langle\omega_c, \omega_c\rangle =\dfrac{1}{2}(A\omega_c, \omega_c)
=\dfrac{1}{2}(A_g\dot{g}, \dot{g}) =\dfrac{1}{2}(M_c, \omega_c)
=\dfrac{1}{2}(M, \dot{g}) \ .
$$

\begin{theorem}[Euler's equation]\label{Thm:EulerEquation}
We have
\begin{equation}\label{Thm:EulerEquation:Eq:Euler}
\dfrac{dM_c}{dt}=\{\omega_c, M_c\} \ .
\end{equation}
\end{theorem}

\begin{proof}
The proof of this Theorem can be found in \cite[Appendix 2, Theorem 2]{[Arnold]}.
\end{proof}

\begin{theorem}[Euler--Arnold equation]\label{Thm:Euler-ArnoldEquation}
We have
\begin{equation}\label{Thm:Euler-ArnoldEquation:Eq:Euler-Arnold}
\dfrac{dM_c}{dt}=\{A^{-1} M_c, M_c\} \ .
\end{equation}
\end{theorem}

\begin{proof}
We notice that $\omega_c=L_{g^{-1}*}\dot{g}=A^{-1}AL_{g^{-1}*}\dot{g}=A^{-1}L_{g^{-1}}^*L_g^*AL_{g^{-1}*}\dot{g}
=A^{-1}L_{g^{-1}}^* A_g g=A^{-1}M_c$. Thus \eqref{Thm:Euler-ArnoldEquation:Eq:Euler-Arnold} follows from
\eqref{Thm:EulerEquation:Eq:Euler}.
\end{proof}

One can see that the evolution of the angular momentum in the body $M_c$ is described by
an ordinary differential equation \eqref{Thm:Euler-ArnoldEquation:Eq:Euler-Arnold} which is the
\textit{the Euler--Arnold equation}. The dynamics of this equation is an equivalent way of forming
the geodesic flows on the Riemannian manifold $G$.

\subsection{Formulation of the system \eqref{Intro:Eq:ABmodel}
as the Euler--Arnold equation.}

Let $G$ be the group of all affine transformations of a line $\ell$ (see \cite{[Molchanov]}). We
can represent $G$ in terms of the following matrices:

$$G=\left\{
g=g_{a,b}=\begin{pmatrix} a&b
\\
0&1
\end{pmatrix}; a>0, b\in \mathbb{R}\right\} \ .$$

The group multiplication is then just matrix multiplications:
$g_{a_2,b_2}g_{a_1,b_1}=g_{a_1a_2, a_2b_1+b_2}$. The inverse is
given by $g_{a,b}^{-1}=g_{\frac{1}{a}, -\frac{b}{a}}$. The identity
element $e=g_{1,0}$.

The Lie algebra
$$T_eG=\left\{\begin{pmatrix}x&y\\0&0\end{pmatrix};
x,y\in \mathbb{R}\right\} \ .$$

If $g\in G$ and $M\in T_eG$ then $L_{g*}M=gM$ and $R_{g*}M=Mg$ in
which the multiplication is understood as matrix multiplications.
This is because we have
$\left.\dfrac{d}{dt}\right|_{t=0}\exp(tM)=gM$ and
$\left.\dfrac{d}{dt}\right|_{t=0}\exp(tM)g=Mg$.

Let us use the inner product $\left(\begin{pmatrix}\xi_1 &
\xi_2\\0&0\end{pmatrix}, \begin{pmatrix}\eta_1 &
\eta_2\\0&0\end{pmatrix}\right)=\xi_1\eta_1+\xi_2\eta_2$ for
$\xi,\eta\in T_eG$. In this way we can identify $T_eG$ with
$T_e^*G$. Let $A:\mathbb{R}^2\rightarrow\mathbb{R}^2$ be the identity matrix. We can
introduce a metric on $G$ via $A$: for any $\xi,\eta\in T_eG$ we
introduce $\langle\xi,\eta\rangle=(\xi,\eta)$.

Let $g=\begin{pmatrix}a&b\\0&1\end{pmatrix}$,
$\eta=\begin{pmatrix}\eta_1&\eta_2\\0&0\end{pmatrix}$ and
$\xi=\begin{pmatrix}\xi_1&\xi_2\\0&0\end{pmatrix}$. Then

$$
Ad_g\eta=g\eta g^{-1}
=\displaystyle{\begin{pmatrix}a&b\\0&1\end{pmatrix}
\begin{pmatrix}\eta_1&\eta_2\\0&0\end{pmatrix}
\begin{pmatrix}\frac{1}{a}&-\frac{b}{a}\\0&1\end{pmatrix}}
=\displaystyle{\begin{pmatrix}\eta_1&-b\eta_1+a\eta_2\\0&0\end{pmatrix}} \ .
$$

By definition $(Ad_g^*\xi, \eta)=(\xi,
Ad_g\eta)=\eta_1\xi_1-b\eta_1\xi_2+a\eta_2\xi_2$. Thus we see that
$Ad_g^*\xi=\begin{pmatrix}\xi_1-b\xi_2 & a\xi_2\\0&0\end{pmatrix}$.

For any $\xi=\begin{pmatrix}\xi_1&\xi_2\\0&0\end{pmatrix}\in T_eG$
we can calculate
$$\exp(t\xi)=\begin{pmatrix}e^{t\xi_1}& \xi_2f(\xi_1)\\
0&1\end{pmatrix}$$ where
$$f(\xi_1)=\left\{\begin{array}{ll}\dfrac{e^{t\xi_1}-1}{\xi_1}
&\text{ if } \xi_1\neq 0 \ ,\\t &\text{ if } \xi_1=0 \
.\end{array}\right.$$

From here it is readily checked that
$$
[\xi,\eta]
\displaystyle{=\left.\dfrac{d}{dt}\right|_{t=0}Ad_{e^{t\xi}}\eta}
\displaystyle{=\left.\dfrac{d}{dt}\right|_{t=0}\begin{pmatrix}\eta_1&
-\dfrac{\xi_2}{\xi_1}(e^{t\xi_1}-1)\eta_1+e^{t\xi_1}\eta_2\\0&0\end{pmatrix}}
\displaystyle{=\begin{pmatrix}0&\xi_1\eta_2-\xi_2\eta_1\\0&0\end{pmatrix}} \ .
$$

Moreover, if $\zeta\in T_e^*G$ we have

\begin{equation}\label{Eq:PoissonBracketAffineLine}
\{\xi,\zeta\}
\displaystyle{=\left.\dfrac{d}{dt}\right|_{t=0}Ad_{e^{t\xi}}^*\zeta}
\displaystyle{=\left.\dfrac{d}{dt}\right|_{t=0}
\begin{pmatrix}\zeta_1-\dfrac{\xi_2}{\xi_1}(e^{t\xi_1}-1)\zeta_2&e^{t\xi_1}\zeta_2\\0&0\end{pmatrix}}
\displaystyle{=\begin{pmatrix}-\xi_2\zeta_2&\xi_1\zeta_2\\0&0\end{pmatrix}} \ .
\end{equation}

\begin{theorem}\label{Thm:Euler-ArnoldEquationAffineLine}
The Euler--Arnold equation for the group $G$ of all affine transformations of a line $\ell$ is equivalent to
\eqref{Intro:Eq:ABmodel}.
\end{theorem}

\begin{proof}
Set $M_c(t)=(M_{c,1}(t), M_{c,2}(t))$, using \eqref{Eq:PoissonBracketAffineLine},
the Euler--Arnold equation \eqref{Thm:Euler-ArnoldEquation:Eq:Euler-Arnold}
in Theorem \ref{Thm:Euler-ArnoldEquation} is given by
$$(\dot{M}_{c,1}(t), \dot{M}_{c,2}(t))=\{M_c(t), M_c(t)\}=(-M_{c,2}^2(t), M_{c,1}(t)M_{c,2}(t)) \ .$$
Set $x(t)=M_{c,2}(t)$ and $y(t)=-M_{c,1}(t)$, then the above equation is
$$(-\dot{y}, \dot{x})=(-x^2, -xy) \ ,$$
which is the same as \eqref{Intro:Eq:ABmodel}.
\end{proof}

We have seen that our unperturbed system \eqref{Intro:Eq:ABmodel} is nothing but the Euler--Arnold equation
for the group $G$ of all affine transformations of a line $\ell$.

\section{Remarks and Generalizations.}\label{Sec:Remark}

1. Let us introduce the elliptic operator
\begin{equation}\label{Rmk:Eq:Operator:Elliptic:Leps}
L_\varepsilon=\dfrac{1}{\varepsilon}\left(-xy\dfrac{\partial}{\partial x}+x^2\dfrac{\partial}{\partial y}\right)-x\dfrac{\partial}{\partial x}-y\dfrac{\partial}{\partial y}
+\dfrac{1}{2}\dfrac{\partial^2}{\partial x^2}+\dfrac{1}{2}\dfrac{\partial^2}{\partial y^2} \ .
\end{equation}
The above elliptic operator can be written as
$$L_\varepsilon=\dfrac{1}{\varepsilon}L_0+L_1 \ ,$$
in which
\begin{equation}\label{Rmk:Eq:Operator:FirstOrder:L0}
L_0=-xy\dfrac{\partial}{\partial x}+x^2\dfrac{\partial}{\partial y} \ ,
\end{equation}
and
\begin{equation}\label{Rmk:Eq:Operator:Elliptic:L1}
L_1=-x\dfrac{\partial}{\partial x}-y\dfrac{\partial}{\partial y}
+\dfrac{1}{2}\dfrac{\partial^2}{\partial x^2}+\dfrac{1}{2}\dfrac{\partial^2}{\partial y^2} \ .
\end{equation}
In this way, the operator $L_0$ degenerates on $x=0$. One can consider a corresponding Cauchy problem
\begin{equation}\label{Rmk:Eq:CauchyProblem}
\dfrac{\partial u^\varepsilon}{\partial t}=L_\varepsilon u^\varepsilon \ , \ u^\varepsilon(0,x,y)=f(x,y) \ ,
\end{equation}
where $f(x,y)$ is a bounded continuous function in $(x,y)\in \mathbb{R}^2$. The solution is represented by
$$u^\varepsilon(t,x,y)=\mathbf{E}_{(x,y)}f(X^\varepsilon_t,Y^\varepsilon_t) \ .$$
By our Theorem \ref{Thm:MainTheorem}, we infer that
$\lim\limits_{\varepsilon\downarrow 0}\mathbf{E}_{(x,y)}f(X^\varepsilon_t,Y^\varepsilon_t)= \lim\limits_{\varepsilon\downarrow 0}\mathbf{E}_{(x,y)}f(0,Y^\varepsilon_t)=\mathbf{E}_{(0,y^\pi(x,y))}f(0,Y_t)$. This gives the following

\begin{corollary}\label{Rmk:Cor:CauchyProblemConvergence}
Let the initial condition $f(x,y)$ be a bounded continuous function of $(x,y)$. Then as $\varepsilon\rightarrow 0$ we have
$u^\varepsilon(t,x,y)\rightarrow u(t,y^\pi(x,y))$ where $u(t,y)$ is the solution of the equation
\begin{equation}\label{Rmk:Cor:CauchyProblemConvergence:Eq:LimitingODE}
\dfrac{\partial u}{\partial t}=\left(\dfrac{1}{2y}-y\right)\dfrac{\partial u}{\partial y}+\dfrac{1}{2}\dfrac{\partial^2 u}{\partial y^2} \ , \ u(0,y)=f(0,y) \text{ for } y\geq 0
\ , \ \dfrac{\partial u}{\partial y}(0+)=0 \ .
\end{equation}
\end{corollary}

\

\begin{figure}[H]
\centering
\includegraphics[height=8cm, width=8cm, bb=0 0 353 361]{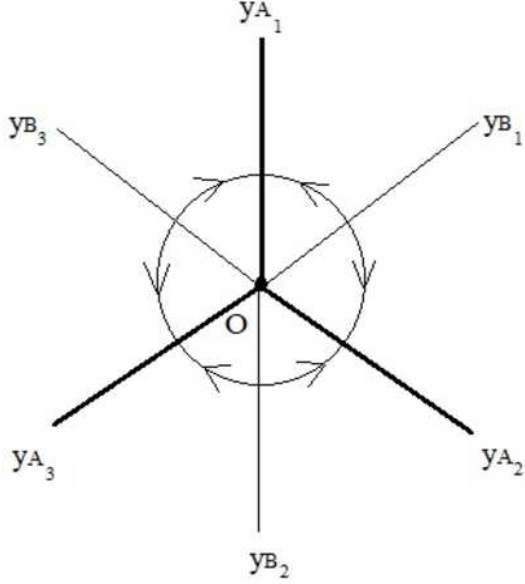}
\caption{A more general problem.}
\label{Fig:GeneralProblem}
\end{figure}

2. One can consider a more general system such as the one shown in Figure \ref{Fig:GeneralProblem}. Here the $3$ axes $Oy_{A_1}$, $Oy_{A_2}$ and $Oy_{A_3}$
consist of stable equilibriums and the other $3$ axes
$Oy_{B_1}$, $Oy_{B_2}$, $Oy_{B_3}$ consist of unstable equilibriums. One can analyze this system in a similar fashion
as we did in this work, so that we expect to see the limiting process as a diffusion process on a tree $\Gamma$ (see \cite{[FW-DiffusionGraph]}).
The tree $\Gamma=Oy_{A_1}\cup Oy_{A_2}\cup Oy_{A_3}$ consists of $3$ edges that are the semi--axes $Oy_{A_1}$, $Oy_{A_2}$, $Oy_{A_3}$. On each edge the limiting process is a Bessel--like process and the interior vertex $O$ is inaccessible. The proof of these facts follows from
the method we adopted in this paper as well as the techniques used in \cite[Chapter 8]{[FWbook]}, \cite{[FW-AMS]}, \cite{[FW-DiffusionGraph]}.
One can first obtain ``localization" type of results as we showed in Lemmas \ref{Lm:HitBackToZeroX:AbsXBiggerThanHalfKappa}, \ref{Lm:HitBackToZeroX:YLessThanMinusDelta}. With such localization results at hand, we then show that the process localized onto $\Gamma$
converges weakly to a diffusion process on the graph $\Gamma$, similarly as we did in the current work.

\

3. If the system \eqref{Eq:ABmodelPerturbedTimeRescaled} do not have the dissipative terms, so that it looks like
\begin{equation}\label{Eq:ABmodelPerturbedTimeRescaledNoDissipation}
\left\{\begin{array}{lll}
d\mathrm{X}_t^\varepsilon=-\dfrac{1}{\varepsilon}\mathrm{X}_t^\varepsilon \mathrm{Y}_t^\varepsilon dt+dW_t^1 & , & \mathrm{X}_0^\varepsilon=x_0 \ ,
\\
d\mathrm{Y}_t^\varepsilon=\dfrac{1}{\varepsilon}(\mathrm{X}_t^\varepsilon)^2dt+dW_t^2 & , & \mathrm{Y}_0^\varepsilon=y_0 \ .
\end{array}\right.
\end{equation}
Then the argument of the Lemmas \ref{Lm:HitClosenessToZeroX:OutsideDelta}--\ref{Lm:TightnessYProcess} and the proof of
Theorem \ref{Thm:MainTheorem} still go through, with
minor changes in the estimates. The limiting $Y$--process will be a process of the form
\begin{equation}\label{Eq:ABmodelPerturbedTimeRescaledNoDissipationLimitYProcess}
d\mathrm{Y}_t=\dfrac{1}{\mathrm{Y}_t}dt+dW_t^2 \ , \ \mathrm{Y}_0=y^\pi(x_0,y_0) \ .
\end{equation}
In particular, this implies that the $\mathrm{Y}_t$ process keeps growing in the direction $Oy_A$. That is to say,
the energy grows in the direction of the stable manifold $Oy_A$. Geometrically, this phenomenon
comes from the fact that the energy constraint given by the conservative flow $b(x,y)=(-xy, x^2)$
provides a positive force around the stable line $Oy_A$. Thus the energy can keep growing at $Oy_A$
due to the random noise. Such a geometric phenomenon might be related
to some problems in $2$--d turbulence (see \cite{[EHSPartialDamping]}).

\

\

\textbf{Acknowledgement.} The author would like to thank Professor
Vladim\'{i}r \v{S}ver\'{a}k from University of Minnesota, USA for fruitful discussions on the formulation of his system \eqref{Intro:Eq:ABmodel}
as the Euler--Arnold equation on the group of all
affine transformations of a line, as well as its relation with fluid mechanics.
He also would like to thank the anonymous referee, Professor Yong Liu from Peking University, Beijing, China
and Professor Yong Ren from Anhui Normal University, Wuhu, Anhui, China for their valuable comments that improve the
first version of this work.

\bibliographystyle{plain}
\bibliography{bibliography}

\end{document}